%% file: Random_walks_in_AH_GAFA.tex
\documentclass{article}
\usepackage{amsfonts, amsmath, amssymb, amscd, amsthm, graphicx,enumerate, import}
\usepackage{hyperref,xypic}
\usepackage[usenames, dvipsnames]{xcolor}
\usepackage{caption, subcaption}

\usepackage{thmtools} 
\usepackage{thm-restate}

\makeatletter
\def\blfootnote{\xdef\@thefnmark{}\@footnotetext}
\makeatother

\newtheorem{thm}{Theorem}[section]
\newtheorem{cor}[thm]{Corollary}
\newtheorem{lem}[thm]{Lemma}
\newtheorem{prop}[thm]{Proposition}

\newtheorem{ques}[thm]{Question}
\newtheorem{conj}[thm]{Conjecture}
\theoremstyle{definition}
\newtheorem{defn}[thm]{Definition}
\theoremstyle{remark}
\newtheorem{rem}[thm]{Remark}

\newtheorem{claim}[thm]{Claim}

\newfont{\eufm}{eufm10}

\newcommand{\e}{\varepsilon }
\renewcommand{\phi}{\varphi}

\renewcommand{\l}{{\ell}}

\newcommand{\Z}{\mathbb Z}
\newcommand{\R}{\mathbb R}

\newcommand{\diam}{\operatorname{diam}}
\newcommand{\Out}{\operatorname{Out}}

\newcommand{\Supp}{\operatorname{Supp}}
\newcommand{\Prob}{\operatorname{Prob}}
\newcommand{\MCG}{\operatorname{MCG}}

\def\Cay {\mathrm{Cay}}

\def\mc {\mathcal}

\begin{document}
\title{Random walks and quasi-convexity in acylindrically hyperbolic groups}
\author{C. Abbott, M. Hull}

\date{}
\maketitle

\begin{abstract}

It is known that every infinite index quasi-convex subgroup $H$ of a non-elementary hyperbolic group $G$ is a free factor in a larger quasi-convex subgroup of $G$. We give a probabilistic generalization of this result. That is, we show that when $R$ is a subgroup generated by independent random walks in $G$, then $\langle H, R\rangle\cong H\ast R$ with probability going to one as the lengths of the random walks go to infinity  and this subgroup is quasi-convex in $G$. Moreover, our results hold for a large class of groups acting on hyperbolic metric spaces and subgroups with quasi-convex orbits. In particular, when $G$ is the mapping class group of a surface and $H$ is a convex cocompact subgroup we show that $\langle H, R\rangle$ is convex cocompact and isomorphic to $ H\ast R$. 

\end{abstract}



\section{Introduction}

 In this paper we study random walks in groups acting on hyperbolic metric spaces and the interaction between  the elements produced by random walks  and fixed subgroups. We primarily focus on using random walks as a way to study the behavior of ``typical" elements and subgroups, both from an algebraic perspective and from the perspective of the dynamics of the action on the hyperbolic metric space. Our general philosophy is that if $H\leq G$ is a subgroup whose orbits are both quasi-convex and sufficiently small compared to the orbits of $G$, then the elements produced by random walks in $G$ will interact as freely as possible with $H$. 

For a non-elementary hyperbolic group $G$, the notion of elements of $G$ interacting freely with a subgroup $H$ goes back to a result stated by Gromov in \cite{Gro} which says that when $H$ is a quasi-convex subgroup of infinite index, then there exists $g\in G$ such that $\langle H, g\rangle\cong H\ast \langle g\rangle$ and this subgroup is quasi-convex. This was proved in the torsion-free case by Arzhantseva \cite{Arz}, and the general case follows by work of Minasyan  \cite{Minasyan3}; see also \cite{Cha}. Similar free product theorems for other classes of groups and subgroups can be found in \cite{AD, AntCum, ArzMin}.

We use random walks to give a probabilistic generalization of these existence theorems which holds for the much larger class of acylindrically hyperbolic groups $G$ and many subgroups $H$ with quasi-convex orbits for an appropriate action of $G$. Random walks in groups acting on hyperbolic metric spaces have recently been studied by a number of people \cite{CalMah, MS, MT, MaherTiozzo18, TaylorTiozzo}. In particular,  Maher--Tiozzo showed (among many other things) that if $G$ has a non-elementary action on a hyperbolic metric space $X$, then a random element of $G$ will act loxodromically on $X$ \cite{MT}. Talyor--Tiozzo prove that subgroups of $G$ generated by $k$ independent random walks are isomorphic to the free group $\mathbb F_k$ and quasi-isometrically embedded in $X$ under the orbit map \cite{TaylorTiozzo}. When the action is also acylindrical, Maher--Sisto proved that these subgroups are geometrically separated in $X$ \cite{MS} and hence hyperbolically embedded in $G$ in the sense of \cite{DGO}. In many situations, we are able to show that if $R$ is the subgroup generated by $k$ independent random walks, then the subgroup generated by a subgroup $H$ with quasi-convex orbits and $R$ will be isomorphic to the free product $H\ast R$ and hence isomorphic to $H\ast \mathbb F_k$ by \cite{TaylorTiozzo}. Moreover, this subgroup $\langle H, R\rangle$ will have quasi-convex orbits in $X$ and will be quasi-isometrically embedded in $X$ and/or geometrically separated in $X$ whenever $H$ is.

Before stating our results precisely, we introduce some terminology and notation. Let $\mu$ be a probability distribution on a group $G$. By a \emph{random walk of length $n$} with respect to $\mu$, we mean a random element $w(n)$ which is equal to the product
\[
w(n)=g_1\cdots g_n,
\]
where the $g_i$ are independent, identically-distributed elements of $G$ with distribution $\mu$. Given a property $P$ and a probability distribution $\mu$, we say that a \emph{random element $r$ of $G$ satisfies $P$} if $\Prob(w(n) \text{ satisfies $P$})\to 1$ as $n\to\infty$. For distributions $\mu_1,\dots,\mu_k$, we say that a \emph{random $k$-generated subgroup $R$ of $G$ satisfies $P$} if $\Prob(\langle w_{1}(n),\dots, w_{k}(n)\rangle \text{ satisfies $P$})\to 1$ as $n\to\infty$, where $w_{1}(n),\dots,w_{k}(n)$ are independent random walks of length $n$ with respect to $\mu_1,\dots,\mu_k$, respectively. Given a sequence $(\mu_i)=(\mu_1, \mu_2,\dots)$ of probability distributions, we say that a \emph{random subgroup $R$ of $G$ satisfies $P$} if the random $k$-generated subgroup with respect to $\mu_1,\dots,\mu_k$ satisfies $P$ for all $k\geq 1$. The canonical example is when $(\mu_i)$ is the constant sequence corresponding to the uniform measure on a finite symmetric generating set for $G$. While our main results are stated under more general assumptions on the sequence of probability distributions, all of our results will hold for this canonical example.

A group $G$ is \emph{acylindrically hyperbolic} if it admits a non-elementary action on a hyperbolic metric space which contains at least one WPD element; see Section \ref{sec:GA} for precise definitions. We call such an action of $G$ a \emph{partially WPD} action. This definition is equivalent to several other common definitions of acylindrical hyperbolicity by \cite{OsinAH}. 
When $G$ has a non-elementary, partially WPD action on a hyperbolic metric space $X$, we fix a basepoint $x_0\in X$ and consider the orbit map $\pi\colon G\to X$ given by $\pi(g)=gx_0$. This map allows us to project a random walk in $G$ to a random walk in $X$. We will need that our probability distributions take bounded steps with respect to this random walk on $X$ and that they ``see" a sufficient amount of the action of $G$ on $X$. We call such distributions \emph{permissible}; see Section \ref{sec:RW} for the precise definition. For now we note that any probability distribution $\mu$ whose support is a finite, symmetric generating set of $G$ will be permissible with respect to any non-elementary, partially WPD action of $G$ on a hyperbolic metric space.

In addition to the subgroup $H$ having quasi-convex orbits, it is also necessary that orbits of $H$ are ``not too large" in $X$; for example, $H$ cannot be a finite-index subgroup of $G$. For this, we require the existence of a loxodromic WPD element which is \emph{transverse to $H$}. This is a loxodromic WPD element whose quasi-axis in $X$ has uniformly bounded intersection with the orbit of any coset of $H$; see Definition \ref{def:transverse}.

Let $\Gamma_\mu$ denote the subgroup of $G$ generated by the support of $\mu$, and let $E(G)$ denote the maximal finite normal subgroup of $G$; this subgroup exists and is unique by \cite{DGO}. We are now ready to state our main theorem.

\begin{restatable}{thm}{main} \label{thm:main1}
Let $G$ be a group with a non-elementary, partially WPD action on a hyperbolic metric space $X$, and let $(\mu_i)$ be a sequence of permissible probability distributions on $G$. Let $H$ be a subgroup of $G$ such that $H\cap E(G)=\{1\}$. Suppose $H$ has quasi-convex orbits in $X$ and there exists a loxodromic WPD element $f\in\cap\Gamma_{\mu_i}$ transverse to $H$. Then a random subgroup $R$ will satisfy $\langle H, R\rangle\cong H\ast R$ and $\langle H, R\rangle$ has quasi-convex orbits in $X$. Moreover, if $H$ is quasi-isometrically embedded in $X$, then $\langle H, R\rangle$ is quasi-isometrically embedded in $X$.
\end{restatable}

The assumption that each $H\cap E(G)=\{1\}$ is necessary to conclude that $\langle H, R\rangle\cong H\ast R$. However it can be be removed at the cost of replacing this free product with an amalgamated product over a finite subgroup; see Remark \ref{finnormsub}. \\

\noindent{\bf Applications to elliptic subgroups.} If $G$ has a non-elementary, acylindrical action on a hyperbolic metric space with $E(G)=\{1\}$ and $H$ is an elliptic subgroup (i.e., $H$ has bounded orbits), then the first author and Dahmani show that there exists $g\in G$ such that $\langle H, g\rangle\cong H\ast \langle g\rangle$ \cite{AD}. When $H$ is elliptic, any loxodromic element of $G$ will be transverse to $H$, so Theorem \ref{thm:main1} implies the following probabilistic generalization of this result.

\begin{cor}\label{cor:mainell}
Let $G$ be a group with a non-elementary, partially WPD action on a hyperbolic metric space $X$, and let $(\mu_i)$ be a sequence of permissible probability distributions on $G$ such that there exists a loxodromic WPD element $f\in\cap\Gamma_{\mu_i}$. Let $H$ be an elliptic subgroup of $G$ such that $H\cap E(G)=\{1\}$. Then a random subgroup $R$ will satisfy $\langle H, R\rangle\cong H\ast R$. 
\end{cor}

There are already a number of applications in this setting; we list some of them below. We say that $\mu$ has \emph{full support} if $\Gamma_\mu=G$. A surface $S$ of genus $g$ with $p$ punctures is called \emph{exceptional} if $3g+p\leq 4$; otherwise $S$ is \emph{non-exceptional}.  See Section \ref{sec:App} for other definitions and notation used in this corollary.

\begin{cor} \label{cor:elliptic}
Suppose that $G$ and $H$ are one of the following.
\begin{enumerate}[(1)]
\item $G$ is a non-elementary hyperbolic group and $H$ is a finite subgroup with $H\cap E(G)=\{1\}$.
\item $G$ is a non-elementary relatively hyperbolic group and $H$ is finite or conjugate into a peripheral subgroup with $H\cap E(G)=\{1\}$.
\item $G$ is the mapping class group of a non-exceptional surface $S$ and $H$ is a subgroup which contains no pseudo-Anosov elements and $H\cap Z(G)=\{1\}$.
\item $G=Out(\mathbb F_n)$ for $n\geq 3$ and $H$ is a subgroup which virtually fixes a free factor of $\mathbb F_n$ up to conjugacy.
\item $G$ is a directly indecomposable right-angled Artin group $A(\Gamma)$ and $H$ is conjugate to a subgroup of $G$ whose support is contained in a subjoin of $\Gamma$.
\item $G=A/Z(A)$ where $A\neq A_1, A_2, I_{2m}$ is an irreducible Artin-Tits group of spherical type and $H$ is a parabolic subgroup.
\item $G=\pi_1(M)$ and $H\leq\pi_1(N)$ where $M$ is a closed, orientable, irreducible, non-geometric 3--manifold and $N$ is a JSJ--component of $M$.
\end{enumerate}
If $(\mu_i)$ is a sequence of finitely supported probability distributions with full support on $G$, then a random subgroup $R$ will satisfy $\langle H, R\rangle\cong H\ast R$.
\end{cor}

We note that the group $G$ in each of the above examples can be replaced with any subgroup whose induced action on the associated hyperbolic space is non-elementary. For example, one could take $G$ to be the Torelli group instead of the whole mapping class group.

When $G$ is acylindrically hyperbolic, then one can always find a non-elementary, partially WPD action of $G$ in which any given finite collection of elements of $G$ will all be elliptic (see, for example, \cite[Lemma 3.18]{H16}). Hence we get the following corollary which shows that random elements of $G$ can be used to show that $G$ satisfies property $P_{naive}$; see \cite{AD}.

\begin{cor}
Let $G$ be acylindrically hyperbolic with $E(G)=\{1\}$, $\mu$ a finitely supported probability distribution on $G$ of full support, and $g_1,\dots, g_m$ fixed elements of $G$. Then a random element $r$ of $G$ will satisfy  $\langle g_i, r\rangle\cong \langle g_i\rangle \ast \langle r \rangle$ for all $1\leq i\leq m$.
\end{cor}

\noindent{\bf Applications to quasi-convex and quasi-isometrically embedded subgroups.}
When the orbits of $H$ are unbounded, it is necessary to find the transverse element $f$ in order to apply Theorem \ref{thm:main1}. If $G$ is hyperbolic, then $f$ is transverse to $H$ if and only if no non-trivial power of $f$ is conjugate into $H$. The existence of such an element is provided by \cite[Proposition 1]{Minasyan2}.  Thus we can prove that Theorem \ref{thm:main1} holds when $G$ is a non-elementary hyperbolic group and $H\leq G$ is an infinite index quasi-convex subgroup. This is a probabilistic generalization of \cite[Theorem 1]{Arz}.

\begin{thm}\label{thm:hypqc}
Let $G$ be a non-elementary hyperbolic group and let $(\mu_i)$ be a sequence of permissible probability distributions on $G$. Let $H$ be an infinite index quasi-convex subgroup of $G$ such that $H\cap E(G)=\{1\}$. Then a random subgroup $R$ will satisfy $\langle H, R\rangle\cong H\ast R$ and $\langle H, R\rangle$ will be an infinite index quasi-convex subgroup of $G$. 
\end{thm}

We are also able to show the analogue of this result in the relatively hyperbolic setting (see Theorem \ref{thm:relhypgps}), with the notion of \emph{relatively quasi-convex} subgroups playing the role of quasi-convex subgroups; see Section \ref{sec:relhyp} for definitions. We point out one interesting application resulting from combining Theorem \ref{thm:relhypgps} with \cite[Corollary 1.3]{Hruska}.
\begin{cor}\label{cor:relqcx}
Let $G$ be a geometrically finite Kleinian group and let $H$ be a geometrically finite subgroup of $G$. Let $(\mu_i)$ be a sequence of finitely supported probability distributions of full support on $G$. Then a random subgroup $R$ satisfies $\langle H, R\rangle\cong H\ast R$ and $\langle H, R\rangle$ is geometrically finite.
\end{cor}

In many cases when $G$ is a group with an non-elementary partially WPD action on a hyperbolic metric space $X$, the subgroups of $G$ which are quasi-isometrically embedded in $X$ are precisely the \emph{stable} subgroups of $G$ in the sense of Durham--Taylor \cite{ABD, AouDurTay, DT, KobManTay}.  We prove an analogue of \cite[Proposition 1]{Minasyan2} in the setting of a stable subgroup $H$ of an acylindrically hyperbolic group $G$ (Theorem \ref{thm:stablesubgroup}). This provides an element $f$ with no non-trivial powers conjugate into $H$, a condition which is necessary for $f$ to be transverse to $H$. When $H$ is a convex cocompact subgroup of a mapping class group of a non-exceptional surface $S$ (see Section \ref{sec:MCG} for definitions), we are able to show that such an element $f$ is indeed transverse to $H$ with respect to the action of $G$ on the curve complex. This is done by exploiting some of the ``hyperbolic-like" features of Teichm\"uller space $\mc{T}(S)$ and comparing the (proper) action of $G$ on $\mc{T}(S)$ to the action of $G$ on the curve complex. Combining this with Theorem \ref{thm:main1} yields the following result.

\begin{thm} \label{thm:MCGCCC}
Let $H$ be a convex cocompact subgroup of $\MCG(S)$ such that $H\cap Z(\MCG(S))=\{1\}$, where $S$ is a non-exceptional surface.  Let $(\mu_i)$ be a sequence of permissible probability distributions on $\MCG(S)$. Then a random subgroup $R$ of $\MCG(S)$ satisfies $\langle H, R\rangle\cong H\ast R$ and $\langle H, R\rangle$ is convex cocompact.
\end{thm}
Note that $\MCG(S)$ has trivial center for any non-exceptional surface $S$ except the closed surface of genus 2, in which case the center contains a single non-trivial element, namely the hyperelliptic involution. If $S$ is a closed surface of genus two and $H$ is a convex cocompact subgroup which contains the hyperelliptic involution, then $\langle H, R\rangle$ will still be convex cocompact, but the free product needs to be replaced with a suitable amalgamated product (see Remark \ref{finnormsub}).

In particular, this implies that (almost) every convex cocompact subgroup of $\MCG(S)$ is a free factor in a larger convex cocompact subgroup:

\begin{cor}\label{cor:concofreefactor}
Let $H$ be a convex cocompact subgroup of $\MCG(S)$ such that $H\cap Z(\MCG(S))=\{1\}$, where $S$ is a non-exceptional surface. Then $\MCG(S)$ contains a convex cocompact subgroup isomorphic to $H\ast \mathbb F_k$ for all $k\geq 1$.
\end{cor}

 It is a well-known open question whether $\MCG(S)$ contains any non-free convex cocompact subgroups \cite[Question 1.7]{FarMos}. If such a subgroup exists, then Corollary \ref{cor:concofreefactor} allows one to construct more examples of non-free convex cocompact subgroups. Note that another way to combine convex cocompact subgroups into potentially new convex cocompact subgroups also appeared recently in \cite{RST}.

We additionally prove the analogue of Theorem \ref{thm:MCGCCC} for convex cocompact subgroups  of $\Out(\mathbb F_n)$ (Theorem \ref{thm:outfn}) by using outer space instead of $\mc{T}(S)$ and the free factor complex instead of the curve complex. We conjecture that the analogue of Theorem \ref{thm:MCGCCC} also holds for stable subgroups of right-angled Artin groups and, more generally, for stable subgroups of hierarchically hyperbolic groups; see Section \ref{sec:App} for precise statements. We note that in all of these situations, $G$ has a non-elementary, partially WPD action on a hyperbolic metric space $X$ such that any subgroup of $G$ which is quasi-isometrically embedded in $X$ is stable in $G$. \\


\noindent{\bf Geometric separation.}
When $G$ has a non-elementary, acylindrical action on a hyperbolic metric space $X$, Maher--Sisto proved that if $R$ is a random subgroup of $G$ then $\langle R, E(G)\rangle$ is \emph{geometrically separated} in $X$ \cite{MS}. We refer to Definition \ref{defn:geomsep} for a precise definition, but loosely speaking a subgroup is geometrically separated in $X$ if the orbits of distinct cosets of that subgroup spread out from each other quickly. We also note that subgroups of $G$ which are both quasi-isometrically embedded in $X$ and geometrically separated are hyperbolically embedded in $G$ in the sense of \cite{DGO}. 

\begin{restatable}{thm}{geomsep}\label{thm:geomsep}
Let $G$ be a group with a non-elementary, partially WPD action on a hyperbolic metric space $X$, and let $(\mu_i)$ be a sequence of permissible probability distributions on $G$. Let $H$ be a subgroup of $G$. Suppose $H$ is quasi-convex and geometrically separated in $X$ and there exists a loxodromic WPD element $f\in\cap\Gamma_{\mu_i}$ transverse to $H$. Then for a random subgroup $R$, $\langle H, R, E(G)\rangle$ is geometrically separated in $X$. In particular, if $H$ is geometrically separated and quasi-isometrically embedded in $X$, then $\langle H, R, E(G)\rangle$ hyperbolically embeds in $G$.
\end{restatable}

This implies that if $E(G)=1$, then every hyperbolic, hyperbolically embedded subgroup of $G$ is a free factor in a larger hyperbolically embedded subgroup of $G$; see Corollary \ref{cor:hyphypembed}. \\

\noindent{\bf Asymptotic freeness.}
In order to get the free product structure in the previous theorems, some assumption about $H$ is necessary. Indeed, if $H$ is a normal subgroup of $G$, then for any $g\in G$ there will be non-trivial relations of the form $g^{-1}hg=k$ for various $h, k\in H$, hence $\langle H, K\rangle$ is not isomorphic to $H\ast K$ for any subgroup $K\leq G$. However, we can show in general that  if $H$ is generated by a finite set of elements, then the length of the shortest relation between the generators of $H$ and the generators of the random subgroup $R$ goes to infinity as the length of the random walks goes to infinity. This can be viewed as an asymptotic form of freeness between finitely generated subgroups of $G$ and random subgroups of $G$. 

Let $\mathbb F(x_1,\dots,x_k)$ be the free group generated by $x_1,\dots, x_k$, and let $W\in G\ast \mathbb F(x_1,\dots,x_k)$. For $g_1,\dots, g_k\in G$, we denote by $W(g_1,\dots, g_k)$  the element of $G$ obtained by replacing each $x_i$ with $g_i$ in $W$.
\begin{restatable}{thm}{maineq}\label{main:eq}
Let $G$ be an acylindrically hyperbolic group with $E(G)=\{1\}$, and let $\mu_1,\dots,\mu_k$ be finitely supported measures of full support on $G$. If $W\in G\ast \mathbb F(x_1,\dots, x_k)$ is non-trivial, then random elements $r_1,\dots, r_k$ generated by $\mu_1,\dots,\mu_k$ will satisfy $W(r_1,\dots, r_k)\neq 1$.
\end{restatable}

This theorem is a generalization of the fact that an acylindrically hyperbolic group $G$ with $E(G)=\{1\}$ is \emph{mixed identity free}, that is for every non-trivial $W\in G\ast \mathbb F(x_1,\dots, x_k)$ there exists $g_1,\dots, g_k$ with $W(g_1,\dots, g_k)\neq 1$ \cite[Corollary 1.7]{HO16}. Note that the proof of \cite[Corollary 1.7]{HO16} is also non-constructive.\\ 

\noindent{\bf Organization.}
The paper is organized as follows. In Section \ref{sec:background} we give some background on hyperbolic metric spaces, group actions, and random walks.  In Section \ref{sec:quasigeod} we prove our main technical theorem, Theorem \ref{t:words}, from which we deduce Theorems \ref{thm:main1} and \ref{main:eq}.  We prove Theorem \ref{thm:geomsep} in Section \ref{sec:geomsep}.  In Section \ref{sec:stablesubgp}, we prove a generalization of \cite[Proposition 1]{Minasyan2} for stable subgroups, and, finally, in Section \ref{sec:App} we give several applications of our main theorems including Corollaries \ref{cor:elliptic}, \ref{cor:relqcx}, and Theorems \ref{thm:hypqc}, \ref{thm:MCGCCC}.\\

\noindent{\bf Acknowledgements.}  The authors thank Joseph Maher, Jason Manning, Bin Sun, and Sam Taylor for useful conversations. We especially thank Sam Taylor for pointing us to \cite[Lemma 4.2]{DowTay}, allowing us to complete the proof of Theorem \ref{thm:outfn}. We also thank Ashot Minasyan for comments on an earlier version of the paper, and we thank the organizers of the special session ``Boundaries and Non-positive Curvature in Group Theory" at the 2018 AMS Spring Southeastern Sectional Meeting where this project was started.  The first author was supported by NSF Award DMS-1803368.


\section{Preliminaries}
\label{sec:background}
\subsection{Hyperbolic metric spaces}

In this section we collect some basic properties of hyperbolic metric spaces. Unless a specific reference is given, the proofs of these properties are either straightforward or can be found in standard references; see for example \cite{BH}.
 
Let $X$ be a metric space. Given three points $x,y,z\in X$, the \emph{Gromov product} is defined to be \[(x\mid y)_z=\frac12\left(d_X(x,z)+d_X(y,z)-d_X(x,y)\right).\]  

We use the following ``thin triangle" definition of a $\delta$--hyperbolic metric space. This is equivalent to several other standard definitions of hyperbolic metric spaces; see for example \cite[Proposition 1.17]{BH}.

\begin{defn} \label{def:thintriangles}
A metric space $X$ is \emph{$\delta$--hyperbolic} if the following holds for any three points $x_0,x_1,x_2\in X$. If $p\in [x_0, x_1]$ and $q\in[x_0, x_2]$ satisfy $d(x_0, p)=d(x_0, q)\leq (x_1\mid x_2)_{x_0}$, then $d(p, q)\leq\delta$.
\end{defn}

A map of metric spaces $f\colon (X,d_X)\to(Y,d_Y)$ is a \emph{$(\lambda,c)$--quasi-isometric embedding} if for all $x,y\in X$, \[\frac1\lambda d_X(x,y)-c\leq d_Y(f(x),f(y))\leq  \lambda d_X(x,y)+c.\]  A \emph{$(\lambda,c)$--quasi-geodesic} is a $(\lambda,c)$--quasi-isometric embedding of an interval $I\subseteq \R$ into $X$, and a \emph{geodesic} is an isometric embedding of $I$ into $X$.  We often conflate geodesics and quasi-geodesic with their images in $X$.

 Given a subset $Y$ of a metric space $X$, let $\mc N_K(Y)$ denote the closed $K$--neighborhood of $Y$ in $X$. Let $d_{Hau}$ denote Hausdorff distance in $X$; that is, for $Y_1, Y_2\subseteq X$, the Hausdorff distance $d_{Hau}(Y_1, Y_2)$ is the infimum of all $K$ such that $Y_1\subseteq \mc{N}_K(Y_2)$ and  $Y_2\subseteq \mc{N}_K(Y_1)$.
 
 A subset $Y\subseteq X$ is \emph{$\sigma$--quasi-convex} if any geodesic in $X$ with endpoints in $Y$ is contained in $\mc{N}_\sigma(Y)$. The subspace $Y$ is called \emph{quasi-convex} if it is $\sigma$--quasi-convex for some $\sigma\geq 0$.

\begin{lem} \label{lem:Morse}
For all $\delta,c\geq 0$ and $\lambda\geq 1$, there is a constant $M=M(\delta,\lambda,c)$ satisfying the following.  Let $X$ be a $\delta$--hyperbolic space,  $q$ a $(\lambda,c)$--quasi-geodesic in $X$, and  $p$ a geodesic with the same endpoints.  Then $d_{Hau}(q, p)\leq M$.
\end{lem}

We say that a constant $M$ as in Lemma \ref{lem:Morse} is a \emph{Morse constant} for $(\lambda,c)$--quasi-geodesics in a $\delta$--hyperbolic space. In some cases, we use the following explicit bound on the Morse constant.

\begin{lem} [{\cite[Theorem~1.1]{GouShc}}] \label{lem:Morsebd}
The Morse constant $M=M(\delta, \lambda, c)$ satisfies \[M\leq 92\lambda^2(c+\delta).\] 
\end{lem}
 
Let $\gamma\colon [0,n]\to X$ be a geodesic with a unit-speed parametrization in a metric space $X$.  For any $K\geq 0$, the \emph{$K$--central segment} of $\gamma$ is the subgeodesic formed by removing open $K$--balls from each endpoint, i.e., $[\gamma(K),\gamma(n-K)]$.  Note that if $K>\frac{n}{2}$, then the $K$--central segment is empty.

\begin{lem} \label{lem:centralseg}
For $i=1,2$, let $p_i\colon [0,n_i]\to X$ be geodesics in a $\delta$--hyperbolic metric space $X$.   Then the $(K+2\delta)$--central segment of $p_1$ is contained in $\mc{N}_{2\delta}(p_2)$, where $K=\max\{d_X(p_1(0),p_2(0)),d_X(p_1(n_1),p_2(n_2))\}$. 
\end{lem}

\begin{lem}
For $i=1,2$, let $q_i\colon [0,n_i]\to X$ be $(\lambda, c)$--quasi-geodesics in a $\delta$--hyperbolic metric space $X$, and let $M$ be corresponding Morse constant. Let $K=\max\{d_X(q_1(0),q_2(0)),d_X(q_1(n_1),q_2(n_2))\}$. Then $d_{Hau}(q_1, q_2)\leq 2\delta+2M+K$.
\end{lem}

Given a path $p$ in a metric space $X$, we denote the length of $p$ in $X$ by $\l(p)$.  Given a sequence of paths $p_1,\dots,p_k$, we denote their concatenation by $p_1\cdot p_2\cdot \ldots \cdot p_k$.

\begin{lem}[{\cite[Lemma 4.2]{Minasyan}}] \label{lem:brokenqgeod}
Let $x_0, x_1,\dots, x_n$ be points in a $\delta$--hyperbolic space $X$ and let $q_i$ be a $(\lambda, c)$--quasi-geodesic from $x_{i-1}$ to $x_i$. Then for any $C_0\geq 14\delta$ and for $C_1=12(C_0+\delta)+c+1$, if $\ell(q_i)\geq \lambda C_1$ and $(x_{i-1}\mid x_{i+1})_{x_i}\leq C_0$, then the concatenation $q_1\cdot \ldots\cdot q_n$ is a $(4\lambda, \frac52M+C_1)$--quasi-geodesic, where $M$ is the Morse constant for $(\lambda,c)$--quasi-goedesics in a $\delta$--hyperbolic metric space. Moreover, if each $q_i$ is a geodesic,  then $q_1\cdot \ldots\cdot q_n$ is a $(2, 2C_1)$--quasi-geodesic.
\end{lem}

The moreover statement can be extracted from the proof of \cite[Lemma 4.2]{Minasyan}; in fact, it follows easily from \cite[Lemma 2.5]{Minasyan}.

\subsection{Group actions}\label{sec:GA}

 Let a group $G$ act by isometries on a metric space $X$, fix a basepoint $x_0$ of $X$, and let $\pi\colon G\to X$ denote the corresponding orbit map, that is, $\pi(g)=gx_0$. A subgroup $H$ of $G$ is  \emph{elliptic} if $\pi(H)$ is a bounded subset of $G$ and \emph{quasi-convex in $X$} if $\pi(H)$ is a quasi-convex subset of $X$. The subgroup $H$ is  \emph{quasi-isometrically embedded in $X$} if $H$ is finitely generated and for some (equivalently, any) finite generating set $S$ of $H$, the restriction $\pi |_{(H, d_S)}$ is a quasi-isometric embedding, where $d_S$ is the corresponding word metric on $H$.   All of these notions are independent of the choice of basepoint $x_0$.

Suppose now that $X$ is a hyperbolic metric space. An element $f\in G$ is called \emph{loxodromic} if $f$ acts as non-trivial translation along a bi-infinite quasi-geodesic axis, which we denote by $\alpha_f$. In this case $f$ has two limit points on $\partial X$, which we denote by $f^{+\infty}$ and $f^{-\infty}$. The action of $G$ on $X$ is called \emph{non-elementary} if $G$ has two loxodromic elements $f_1$ and $f_2$ such that $\{f_1^{+\infty}, f_1^{-\infty}\}\cap\{f_2^{+\infty}, f_2^{-\infty}\}=\emptyset$.

A loxodromic element $f$ is called a \emph{WPD element} if for all $\kappa>0$, there exists $N$ such that for any $x\in \alpha_f$,
\[
\left|\{g\in G\;|\; d(x, gx)\leq \kappa,\, d(f^Nx, gf^Nx)\leq\kappa\}\right|<\infty.
\]
In other words, $G$ contains finitely many elements which almost fix points that are far apart along the axis of $f$. A loxodromic WPD element $f$ is contained in a unique, maximal virtually cyclic subgroup $E_G(f)$ which is equal to the setwise stabilizer of $\{f^{+\infty}, f^{-\infty}\}$. There is a subgroup $E_G^+(f)$ of $E_G(f)$ with index at most $ 2$ that fixes $\{f^{+\infty}, f^{-\infty}\}$ pointwise.  We denote $E_G(f)$ and $E_G^+(f)$ by $E(f)$ and $E^+(f)$, respectively, when $G$ is understood.

The action of $G$ on $X$ is called a \emph{WPD action} if every loxodromic element is a WPD element, and a \emph{partially WPD action} if $G$ contains at least one loxodromic WPD element. The action of $G$ on $X$ is called \emph{acylindrical} if for all $\kappa>0$, there exists $N$ and $R$ such that for all $x, y\in X$ with $d(x, y)\geq R$,
\[
\left|\{g\in G\;|\; d(x, gx)\leq \kappa,\, d(y, gy)\leq\kappa\}\right|\leq N.
\]

Clearly every acylindrical action is WPD, and every WPD action with at least one loxodromic element is partially WPD. It turns out that if $G$ has a non-elementary, partially WPD action on a hyperbolic metric space, then in fact it also has a non-elementary acylindrical action on a (possibly different) hyperbolic metric space \cite{OsinAH}. Such groups are called \emph{acylindrically hyperbolic}. For background on the theory of acylindrically hyperbolic groups and various other equivalent definitions, as well as numerous examples, we refer to \cite{DGO, OsinAH}.

Finally, we give the definition of element being transverse to a subgroup. This notion is one of the key elements in the proofs of our main theorems.
\begin{defn}\label{def:transverse}
Let a group $G$ act on a hyperbolic metric space $X$,  and let $S$ be a subset of $G$. A loxodromic element $f$ is \emph{transverse to $S$} if $f$ has a quasi-geodesic axis $\alpha_f$ in $X$  such that for all $K>0$, there exists $L\geq 0$ such that $\diam(\alpha_f\cap \mc{N}_K(g\pi(S)))\leq L$ for all $g\in G$.
\end{defn}

\subsection{Random walks}\label{sec:RW}

Let $\mu$ be a probability distribution on a group $G$. We denote the support of $\mu$ by $\Supp(\mu)$ and the semi-group generated by the support of $\mu$ by $\Gamma_\mu$. If $\Gamma_\mu$ is in fact a subgroup of $G$, then $\mu$ is called \emph{reversible}. We say $\mu$ is \emph{countable} if $\Supp(\mu)$ is countable, $\mu$ is \emph{finitely supported} if $\Supp(\mu)$ is finite, and  $\mu$ has \emph{full support} if $\Gamma_\mu=G$. Given a fixed action of $G$ on a hyperbolic metric space $X$, the probability distribution $\mu$ is \emph{bounded} if some (equivalently, every) orbit of $\Supp(\mu)$ is a bounded subset of $X$, \emph{non-elementary} if the action of $\Gamma_\mu$ on $X$ is non-elementary, and \emph{WPD} if  $\Gamma_\mu$ contains at least one loxodromic WPD element.

Given a reversible, non-elementary, WPD probability distribution $\mu$ on $G$, there exists a unique, maximal finite subgroup of $G$ normalized by $\Gamma_\mu$ \cite[Lemma 5.5]{H16}; see also \cite[Proposition 1.14]{MaherTiozzo18}. We denote this subgroup by $E_G(\mu)$, or just $E(\mu)$ when $G$ is understood.  We note that $E(\mu)$ will always contain the maximal finite normal subgroup of $G$, which we denote by $E(G)$.

\begin{defn}
We say $\mu$ is \emph{permissible} (with respect to $X$) if it is bounded, countable, reversible, non-elementary, WPD, and $E(\mu)=E(G)$. 
\end{defn}

 Note that for the canonical example when the support of $\mu$ is a finite symmetric generating set of $G$, $\mu$ will be finitely supported, hence countable and bounded for any action of $G$. In addition, such $\mu$ will have full support and hence be non-elementary and WPD for any non-elementary, partially WPD action of $G$. The fact that $\Gamma_\mu=G$ also implies that $E(\mu)=E(G)$. In particular, a finitely supported probability distribution of full support will be permissible with respect to any non-elementary, partially WPD action of $G$ on a hyperbolic metric space.

Throughout this section, we fix a group $G$ with a non-elementary, partially WPD action on a $\delta$--hyperbolic metric space $X$ such that $E(G)=\{1\}$. We also fix a  permissible probability distribution $\mu$ on $G$ and let $w(n)$ be the random walk of length $n$ associated to $\mu$.  If $n$ is fixed, then we will simply use $w$.  When we consider a sequence $(\mu_i)$ of permissible probability distributions, we let $w_i(n)$ (or simply $w_i$, if $n$ is fixed) denote the  random walk of length $n$ associated to $\mu_i$.  Let $x_0\in X$ be a fixed basepoint.  We will state all of the results in this section under these assumptions; many hold in greater generality, but we will not need the full statements here.

\begin{thm} [{\cite[Corollary~11.5]{MaherTiozzo18}}] \label{thm:lox}
The probability that $w$ is loxodromic and WPD with $E(w)=E^+(w)=\langle w\rangle$ tends to one as $n$ tends to infinity.
\end{thm}

As shown in \cite{MT}, the limit $D=\lim_{n\to\infty}\frac1n d_X(x_0,w(n)x_0)$ exists and is positive almost surely.  We say that $D$ is the \emph{drift} of the random walk.  
 The following result describes how far a random element moves the basepoint $x_0\in X$ in terms of its drift. 
\begin{thm} [{\cite[Theorem~1.2]{MT}}] \label{thm:drift}
 There is a positive drift constant $D$ such that for any $\e>0$ there are constants $K>0$ and $c<1$ depending on $\mu$ and $\e$ such that for all $n$,
	\[\Prob\big((1-\e)Dn\leq d(x_0,wx_0)\leq (1+\e)Dn\big)\geq 1-Kc^n.\] \end{thm}

Let $\gamma$ denote a geodesic from $x_0$ to $wx_0$ in $X$,  and let $\alpha$ be the axis of $w$, assuming $w$ is loxodromic. In the case that we have a multiple random walks $w_i$ of length $n$ associated to probability distributions $\mu_i$ with drifts $D_i$, we use $\gamma_i$ and $\alpha_i$, respectively, and we also let $D=\min\{D_i\}$.

An important tool in the study of random walks on hyperbolic spaces with a $G$--action is matching estimates.  The following definition of matching is from Maher--Sisto \cite{MS}.
\begin{defn}
Two geodesic $p$ and $q$ in $X$ have an \emph{$(A, B)$--match} if there are subgeodesics $p'\subseteq p$ and $q^\prime\subseteq q$ of length at least $A$ and an element $g\in G$  such that $d_{Hau}(gp', q^\prime)\leq B$.

We say a geodesic $p$ has an \emph{$(A,B)$--self-match} if there are (not necessarily disjoint) subgeodesics $p',p''$ of $p$ of length at least $A$ and an element $g\in G\setminus\{1\}$ such that $d_{Hau}(gp',p'')\leq B$.  
\end{defn}

We note that the term ``self-match" is used slightly differently in \cite{MaherTiozzo18} where they require that $p'$ and $p''$ be disjoint.

\begin{prop}  [{\cite[Proposition~1.5]{DH}, \cite[Proposition 7.5]{MaherTiozzo18}}]\label{thm:matching}
There is a constant $K_0$ depending only on $\delta$ such that for all $K\geq K_0$ and any $\e>0$, the probability that $w$ is loxodromic with axis $\alpha$, and $\gamma$ and $\alpha$ have a $((1-\e)Dn,K)$--match tends to one as $n\to\infty$.
\end{prop}

The following lemma bounds the length of a self-match that can occur in the geodesic $\gamma$.

\begin{lem}\label{lem:gamma_nmatch} 

 Let $0<\e<1$, $K\geq 0$, and  $\eta$ be a subpath of $\gamma$ with $\l(\eta)\geq \e Dn$.  Then the probability that there exists an element $h\in G\setminus\{1\}$ such that $h\eta\subset N_K(\gamma)$ approaches 0 as $n\to \infty$.
\end{lem}

\begin{proof}
Theorem \ref{thm:lox} implies that $E(w)=\langle w\rangle$ with probability approaching 1 as $n\to\infty$.
Let $\eta$ be a subpath of $\gamma$ with $\l(\eta)\geq \e Dn$, and suppose there exists a constant $K$ and an element $h\in G\setminus\{1\}$ such that $h\eta\subset N_K(\gamma)$.  Notice that this implies that $h\not\in\langle w\rangle$.  Fix $\e'\in(1-\e,1)$, and let $K_0$ be given by Proposition \ref{thm:matching}.  With probability approaching 1 as $n\to \infty$, $\gamma$ and $\alpha$ have an $(\e' Dn,K_0)$--match by Proposition \ref{thm:matching}, and so it follows from our choice of $\e'$ that there is a subsegment $\eta'$ of $\eta$ satisfying $\l(\eta')\geq(\e'+\e-1) Dn$ that is contained in the $K_0$--neighborhood of $\alpha$.  Thus there is a subpath of $\alpha$ of length at least $(\e'+\e-1) Dn-2K_0$ whose image under $h$ is contained in the $(2K_0+K)$--neighborhood of $\alpha$.  Therefore there is an $\e''>0$ such that  $\alpha$ has a $(\e''Dn,2K_0+K) $--self-match for all sufficiently large $n$.  As $h\not\in \langle w\rangle =E(w)$,  the probability that this occurs approaches 0 as $n\to\infty$ by \cite[Proposition~11.7]{MaherTiozzo18}.
\end{proof}

\begin{lem}\cite[Corollary 8.12]{MaherTiozzo18}\label{lem:matching}
There exists a constant $K_0$ such that for all $K\geq K_0$ the following holds. Let $w_1$ and $w_2$ be random walks of length $n$ with respect to permissible probability distributions. Then for any $0<\e<1$, the probability that $\gamma_1$ and $\gamma_2$ have an $(\e Dn, K)$--match goes to $0$ as $n\to\infty$.
\end{lem}

We note that \cite[Corollary 8.12]{MaherTiozzo18} is stated for disjoint subpaths of a single random walk, but the same proof shows the above lemma with only the obvious changes.

\begin{lem}[{\cite[Lemma 7.7]{MaherTiozzo18}}]\label{lem:matchwithf}
 Suppose $\alpha_f$ is the axis of a loxodromic WPD element $f\in \Gamma_\mu$. Then there exists a $K_0$ such that for all $\e>0$, for any $K\geq K_0$, and for any $L\geq 0$, the probability that any subpath $\eta$ of $\alpha$ with $\l(\eta)\geq \e Dn$ has an $(L, K)$--match with a translate of $\alpha_f$ goes to 1 as $n\to\infty$.

\end{lem}

Finally, we show that when there exists an element $f$ that is transverse to $H$, the geodesic $\gamma$ does not have a long subpath contained in a neighborhood of $H$. Recall that $\pi\colon G\to X$ denotes the orbit map defined by $\pi(g)=gx_0$.

\begin{lem}\label{matchwithH}
Let  $S$ be a subset of $G$ such that $\pi(S)$ is $\sigma$--quasi-convex in $X$. Suppose there exists a loxodromic WPD element $f\in \Gamma_\mu$ which is transverse to $S$. Then there exists a constant $K_0$ such that for any $K\geq K_0$ and any $0<\e<1$, the probability that $\gamma$ has a subpath of length $\e Dn$  contained in the $K$--neighborhood of a translate of $\pi(S)$ goes to $0$ as $n\to\infty$.
\end{lem}

\begin{proof}
Let $K_0$ be as in Lemma \ref{lem:matchwithf}, let $K_1=\max\{2\delta, K_0\}$, and fix a constant $K\geq K_0$.   Let $M$ be the Morse constant for $(1,2K)$--quasi-geodesics in $X$.  Suppose that $\pi(S)$ is $\sigma$--quasi-convex in $X$.

Fix $0<\e<1$, $0<\e'<\e$, and $1-(\e-\e')<\e''<1$.  Fix sufficiently large $n\in \mathbb N$, and let $w$ be the random walk of length $n$ associated to the permissible probability distribution $\mu$.   We assume that $w$ is loxodromic with axis $\alpha$ and that each of the following holds:
\begin{enumerate}
\item $(1-\e')Dn\leq d(x_0,wx_0)\leq (1+\e')Dn$.
\item $\gamma$ and $\alpha$ have an $(\e''Dn, K_1)$--match.
\end{enumerate}

Both of these hold with probability approaching one as $n\to\infty$ by Theorem \ref{thm:drift} and Proposition \ref{thm:matching}.

Suppose that $\gamma$ has a subpath of length $\e Dn$ which is contained in a $K$--neighborhood of $g\pi(S)$ for some $g\in G$. Then we can choose $a,b\in g\pi(S)$ and $a', b'\in\gamma$ such that $d(a,a')\leq K$, $d(b,b')\leq K$, and $d(a',b')\geq \e Dn$. Let $q$ be a geodesic from $a$ to $b$.  Since $g\pi(S)$ is $\sigma$--quasi-convex, we have $q\subseteq N_{\sigma}(g\pi(S))$. Let $\eta$ be the subpath of $\gamma$ from $a'$ to $b'$. For all sufficiently large $n$, the path that is the concatenation of $[a,a']$, $\eta$, and $[b',b]$ is a $(1,2K)$--quasi-geodesic, and therefore $\eta\subseteq N_{M}(q)$.  

Now $\gamma$ has a $(\e'' Dn, K_1)$--match with $\alpha$. Since $\l(\gamma)\leq (1+\e')Dn$ and $\l(\eta)\geq \e Dn$, it follows that $\eta$ must have subpath of this match of length at least $\e''Dn-((1+\e')Dn-\e Dn)=(\e''-(1-(\e-\e')))Dn$. Our choice of $\e''$ ensures that $\e''':=\e''-(1-(\e-\e'))>0$. 

Since $\eta\subset N_{M}(q)$, we get that $q$ has a $(\e''' Dn-2M, K_1+M)$--match with $\alpha$. Let $\beta$ be the corresponding subpath of $\alpha$. For any $L$ (fixed with respect to $n$), Lemma \ref{lem:matchwithf} allows us to assume that $\beta$ has a $(L, K_0)$--match with $\alpha_f$, and so $q$ has a $(L, K_0+K_1+M)$--match with $\alpha_f$.  Thus there is an element $g'\in G\setminus\{1\}$ and a subpath $\kappa\subseteq g'\alpha_f$ with $\l(\kappa)\geq L$ such that $\kappa\subseteq N_{K_0+K_1+M}(q)$, and hence $\kappa\subseteq N_{K_0+K_1+M+\sigma}(g\pi(S))$. But for sufficiently large $L$, this will contradict the fact that $f$ is transverse to $\pi(S)$.  See Figure~\ref{fig:qcxity}.
\begin{figure}
\def\svgwidth{2.5in}  
  \centering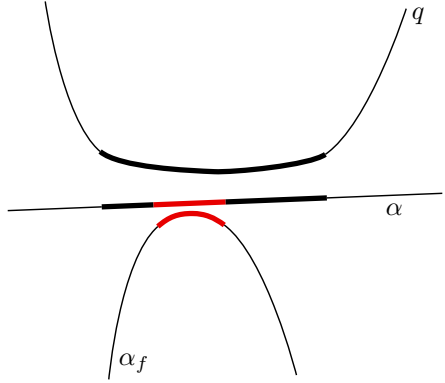 \\
	\caption{The bold black subpaths have length at least $\e'''Dn-2M$ and are within $K_1+M$ of each other, while the red bold subpaths have length at least $L$ and are within $K_0$ of each other.} 
	\label{fig:qcxity}
\end{figure}
\end{proof}

\section{Quasi-geodesic words}\label{sec:quasigeod}

Let $G$ be a group acting on a metric space $X$, and let $W=g_1\cdots g_n$ be a word in the elements of $G$.  We say that $W$ \emph{represents} the element $g_1\dots g_n\in G$ and $W$ is \emph{reduced} if $g_i\neq g_{i+1}^{-1}$ for $1\leq i<n$. Let $\|W\|$ denote the length of the word $W$, i.e., $\|W\|=n$. By a \emph{path labeled by $W$ in $X$ based at $x_0\in X$}, we mean a path from $x_0$ to $g_1\cdots g_nx_0$ which is a concatenation of geodesics of the form
\[
[x_0, g_1x_0]\cdot[g_1x_0, g_1g_2x_0]\cdot\ldots\cdot[g_1\cdots g_{n-1}x_0, g_1\cdots g_nx_0].
\]

We will refer to $g_i$ as the \emph{label} of the subpath $[g_1\cdots g_{i-1}x_0, g_1\cdots g_ix_0]$; subpaths labeled by subwords of $W$ are defined similarly. Since $G$ is acting by isometries, the length of $[g_1\cdots g_{i-1}x_0, g_1 \cdots g_ix_0]$ is equal to $d(x_0, g_ix_0)$ and $(g_1\cdots g_{i-1}x_0\mid g_1\cdots g_{i+1}x_0)_{g_1\cdots g_{i}x_0}= (g_{i}^{-1}x_0\mid g_{i+1}x_0)_{x_0}$.

For the rest of this section, we let $G$ be a group with a non-elementary, partially WPD action on a $\delta$--hyperbolic metric space $X$ and  $(\mu_i)$  a sequence of permissible probability distributions on $G$.  Let $w_{i}$ denote the element generated by a random walk of length $n$ with respect to $\mu_i$.  Fix a basepoint $x_0\in X$, and let $\pi\colon G\to X$ denote the orbit map $\pi(g)=gx_0$.  Let  $\gamma_i$ denote a geodesic $[x_0, w_{i}x_0]$,  let $D_i$ denote the drift of the random walk with respect to $\mu_i$ for each $i$, and let $D=\min D_i$.

\begin{lem}\label{lem:t3paths}
 Let $S$ be a subset of $G$ such that $\pi(S)$ is $\sigma$--quasi-convex in $X$, and suppose there exists a loxodromic WPD element $f\in \cap\Gamma_{\mu_i}$ which is transverse to $S$. Then for any $\e'>0$, the following holds with  with probability going to one as $n\to\infty$. For any $s\in S$, and any $1\leq i\leq k$, the path in $X$ labeled by $sw_{i}$ based at $x_0$ is a $(2, c')$--quasi-geodesic where $c'=24\e' Dn+24\delta+2$.
\end{lem}

\begin{proof}  
 Fix $K=\delta+\sigma$, $0<\e'<\frac12$, and  $C_1=12(\e' Dn+\delta)+1$.  Suppose that $\gamma_i$ has no subpath of length $\e' Dn$ which is contained in the $K$--neighborhood of a translate of $\pi(S)$ and that $\gamma_i$ has length greater than $C_1$. Each of these occurs with probability approaching one as $n\to\infty$ by Lemma \ref{matchwithH} and Theorem \ref{thm:drift}, respectively.

Let $s\in S$, and let $p=[x_0, sx_0]\cdot[sx_0, sw_ix_0]$. We first suppose that $d(x_0, sx_0)\geq C_1$. In this case, we must have $(x_0\mid sw_ix_0)_{sx_0}\leq \e' Dn$. Otherwise, the initial subpath of  $[sx_0, sw_ix_0]$ of length $\e' Dn$ would be contained in the $\delta$--neighborhood of $[x_0, sx_0]$, and hence in the $(\delta+\sigma)$--neighborhood of $\pi(S)$, which contradicts our initial assumption. Thus, in this case, we can apply Lemma \ref{lem:brokenqgeod}, which gives that $p$ is a $(2, 2C_1)$ quasi-geodesic.

Now suppose that $d(x_0, sx_0)\leq C_1$. Then $p$ is the concatenation of a path of length at most $C_1$ and a geodesic, so $p$ is a $(1, 2C_1)$--quasi-geodesic.
\end{proof}

The following is our main technical result. 

\begin{thm}\label{t:words}
Let $G$ be a group with a non-elementary, partially WPD action on a $\delta$--hyperbolic metric space $X$ and let $(\mu_i)$ be a sequence of permissible probability distributions on $G$. Suppose $E(G)=\{1\}$, and let $S$ be a subset of $G\setminus\{1\}$ such that $\pi(S)$ is $\sigma$--quasi-convex. Suppose also that there exists a loxodromic WPD element $f\in \cap\Gamma_{\mu_i}$ which is transverse to $S$. Then there exists a constant $c=c(n,\delta)$ such that the following holds with probability approaching one as $n\to\infty$.  Given any reduced word $W$ in the alphabet $S\cup\{w_{1},\dots,w_{k}\}$ with no consecutive letters belonging to $S$, the path in $X$ based at $x_0$ and labeled by $W$ is a $(8, c)$--quasi-geodesic. In particular, $W$ represents a non-trivial element of $G$. 
\end{thm}

\begin{rem}\label{rem:c}
The constant $c$ can be chosen to be the sum of  a constant multiple of $\e Dn$ and a constant depending only on $\delta$ for any sufficiently small $\e>0$.
\end{rem}

\begin{proof}

Fix $\e>0$ and $\e>\e'>0$ sufficiently small. Let $M$ be the Morse constant for $(2, c')$--quasi-geodesics where $c'$ is given by Lemma \ref {lem:t3paths}. Note that $c'$ and $M$ can each be taken to be the sum of a constant multiple of $\e' Dn$ and a constant depending only on $\delta$ by Lemma \ref{lem:t3paths} and Lemma \ref{lem:Morsebd}, respectively. 
We assume the following are satisfied for each $1\leq i\leq k$.  Recall that $D=\min\{D_i\}$.

\begin{enumerate}[(a)]
\item  $(1-\e)Dn\leq (1-\e)D_in\leq d(x_0,w_{i}x_0)\leq(1
+\e)D_in$.
\item $\gamma_i$ has no $(\frac{\e}{2} Dn, 2\delta+2M)$--self-match.
\item $\gamma_i$ has no $(\frac{\e}{2} Dn, 2\delta+2M)$--match with $\gamma_j$ for any $j\neq i$.
\item $\gamma_i$ has no subpath of length $\frac{\e}{2} Dn$ contained in the  $(2\delta+M+\sigma)$--neighborhood of a translate of  $\pi(S)$.
\item For any $s\in S$, the path labeled by $sw_{i}$ based at $x_0$ is a $(2, c')$--quasi-geodesic.
\end{enumerate}

Properties (a) and (e) hold with probability approaching $ 1$ as $n\to\infty$ by Theorem \ref{thm:drift} and Lemma \ref{lem:t3paths}, respectively. Suppose now that $\gamma_i$ and $\gamma_j$ have a $(\frac{\e}{2} Dn, 2\delta+2M)$--match. Let $\eta$ be the corresponding subpath of $\gamma_i$, that is, $\eta$ is a subpath of length at least $\frac\e2 Dn$ and, for some $g\in G$, we have that $g\eta$ is contained in  the $(2\delta+2M)$--neighborhood of $\gamma_j$. By choosing $\e'$ small enough compared to $\e$,  we can assume that $\frac{\e}{2} Dn> 3(4\delta+2M)$ for all sufficiently large $n$. This implies that  the $(4\delta+2M)$--central segment of $\eta$ will have length at least $\frac{\e}{6} Dn$. By Lemma \ref{lem:centralseg}, the $(4\delta+2M)$--central segment of $\eta$ is contained in the $2\delta$ neighborhood of $\gamma_j$. This implies that $\gamma_i$ and $\gamma_j$ have a $(\frac{\e}{6} Dn, 4\delta)$--match. The probability that such a match exists approaches $0$ as $n\to\infty$ by Lemma \ref{lem:matching}. Hence (c) holds with probability approaching $ 1$ as $n\to\infty$. A similar argument together with Lemma \ref{lem:gamma_nmatch} and Lemma \ref{matchwithH}, respectively, shows that (b) and (d) each hold with probability approaching $ 1$ as $n\to\infty$.
 

Let $C_1=12(\e Dn+4M+\delta)+c'+1$. Note that $C_1$ is at most the sum of a constant multiple of $\e Dn$ and a constant multiple of $\delta$, hence by choosing $\e$  sufficiently small, we have $(1-\e) Dn>2C_1$ for all sufficiently large $n$. We assume this holds, and also that $\e Dn+4M>14\delta$.

If a path labeled by a word $U$ is a quasi-geodesic, the same is true for any subword of $U$. Hence we can assume without loss of generality that the last letter of $W$ is not an element of $S$. Let $q$ be a path labeled by $W$ in $X$ based at $x_0$, and let $q=q_1\cdot\dots \cdot q_m$ such that $q_l$ is one of the following types for each $1\leq l\leq m$:

{\bf Type 1:} $q_l$ is a subpath labeled by $w_{i}^{\pm 1}$ for some $1\leq i\leq k$.

{\bf Type 2:} $q_l$ is a subpath labeled by $sw_{i}^{\pm 1}$, for some $1\leq i\leq k$ and some $s\in S$.

 Note that type 1 subpaths are geodesics of length at least $2C_1$ and type 2 subpaths are $(2, c')$--quasi-geodesics of length at least $2C_1$.

Let $x_{l-1}=(q_l)_-$ for $l=1,\dots, m$, and let $x_m=q_+$.  We will show that for all $1\leq l<m$ the Gromov products satisfy 

\begin{equation}\label{eqn:Gprodbd}(x_{l-1}\mid x_{l+1})_{x_l}< \e Dn+4M.\end{equation}

Assume for the sake of contradiction that $(x_{l-1}\mid x_{l+1})_{x_l}\geq \e Dn+4M$ for some $1\leq l<m$. Then the terminal subpath of $[x_{l-1}, x_l]$ of length $\e Dn+4M$ and the initial subpath of $[x_l, x_{l+1}]$ of length $\e Dn+4M$ are contained in the $\delta$--neighborhoods of each other. Since $d_{Hau}([x_{l-1}, x_l], q_l)\leq M$, this implies that the terminal subpath of $q_l$ of length $\e Dn+3M$ and the initial subpath of $q_{l+1}$ of length $\e Dn+3M$ are contained in the $(\delta+2M)$--neighborhoods of each other. We will use matching estimates to show that this is a contradiction for each possible label of $q_l$ and $q_{l+1}$.
We say that $x_l$ has \emph{type $(a, b)$} if $q_l$ has type $a$ and $q_{l+1}$ has type $b$.

{\bf Case 1:} $x_l$ has type $(1, 1)$ or $(2, 1)$. 

In both of these cases, the terminal subpath of $q_l$ is a translate of some $\gamma_{i}^{\pm 1}$ and  the initial subpath of $q_{l+1}$ is a translate of some $\gamma_j^{\pm 1}$. Hence $\gamma_i$ and $\gamma_j$ have a $(\e Dn+3M, \delta+2M)$--match, contradicting either (b) or (c) depending on whether $i=j$ or $i\neq j$.

{\bf Case 2:} $x_l$ has type $(1, 2)$ or $(2, 2)$.

Let $q_{l+1}=\beta_1\cdot\beta_2$, where $\beta_1$ is labeled by an element $s\in S$ and $\beta_2$ is labeled by $w_{i}^{\pm 1}$  (see Figure 2). Let $\eta$ be the initial subsegment of $[x_l, x_{l+1}]$ of length $\e Dn+4M$. Recall that this means that the terminal subpath of $q_l$ of length $\e Dn+3M$ is contained in the $(\delta+M)$--neighborhood of $\eta$. For all sufficiently large $n$, this terminal subpath of $q_l$ must be a subpath of a translate of some $\gamma_j$.

It follows from applying thin triangles (Definition \ref{def:thintriangles}) that $\eta$ can be decomposed as $\eta=\eta'\cdot\eta''$, where $\eta'\subseteq \mc{N}_{\delta}(\beta_1)$ and $\eta''\subseteq \mc{N}_{\delta}(\beta_2)$, with $\eta''$ possibly empty. Suppose that $\l(\eta')\geq \frac12\l(\eta)=\frac\e2 Dn+2M$. Then the terminal subpath of $q_l$ of length $\frac\e2 Dn+M$  is contained  in the $(2\delta+M+\sigma)$--neighborhood of a translate of $\pi(S)$, contradicting (d).

Now suppose that $\l(\eta')< \frac12 \l(\eta)$, in which case $\l(\eta'')\ge\frac\e2 Dn+2M$. Then  the terminal subpath of $q_l$ has a subpath of length $\frac\e2 Dn$ which is contained in the $(2\delta+M)$--neighborhood of a subpath of $\beta_2$, which contradicts either (b) or (c) depending on whether $i=j$ or $i\neq j$.

\begin{figure}
\def\svgwidth{2.5in}  
  \centering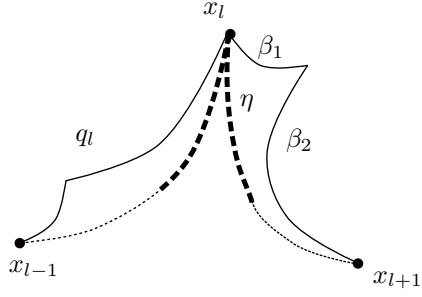 \\
	\caption{Case 2, when $x_l$ has type (2,2).  The solid paths are $q_l$ and $q_{l+1}=\beta_1\cdot\beta_2$, and the dotted paths are the geodesics $[x_{l-1},x_l]$ and $[x_l,x_{l+1}]$.  The bold subgeodesics have length at least $\e Dn+4M$ and are within $\delta$ of each other, leading to a contradiction.} 
	\label{fig:CaseB}
\end{figure}

We have shown that (\ref{eqn:Gprodbd}) holds. Hence we may apply Lemma \ref{lem:brokenqgeod} with $\lambda=2$, $c=c'$, $C_0=\e Dn+4M$, and $C_1=12(C_0+\delta)+c'+1$, which gives that $q$ is a $(8, \frac52M+C_1)$--quasi-geodesic.

If $W$ consists of a single letter $s$ for some $s\in S$, then $s\neq 1$ by assumption. So we may suppose $W$ contains $w_{i}$ for some $i$.  By assumption,  $\l(q)\geq\l(\gamma_i)\geq (1-\e)Dn$.  Thus

\[
d(x_0, x_m)\geq\frac1{8}\l(q)-\frac52M-C_1\geq\frac1{8}\left((1-\e)Dn-\frac52M-C_1\right),
\]
 which is positive for all sufficiently large $n$.  It follows that $d(x_0, x_m)>0$, and so $W$ represents a non-trivial element of $G$.
\end{proof}

We now prove  Theorem \ref{thm:main1}, whose statement we recall for the convenience of the reader.

\main*

\begin{proof}

Assume first that $E(G)=\{1\}$; in this case, the theorem follows easily from Theorem \ref{t:words}.

Now consider the case where $E(G)\neq 1$. Let $\overline{G}=G/E(G)$, and let $\overline{H}$, $\overline{f}$ be the images of $H$ and $f$ in $\overline{G}$. Note that $E(\overline{G})=\{1\}$. Let $\overline{X}$ be the quotient of $X$ obtained by identifying $x$ and $y$ whenever there exists $g\in E(G)$ such that $gx=y$. Since $E(G)$ acts trivially on $\overline{X}$ there is an induced action of $\overline{G}$ on $\overline{X}$. If the action of $G$ on $X$ is cobounded then the quotient map $X\to\overline{X}$ is $G$--equivariant quasi-isometry. In general, we will have that that  the restriction of the quotient map to any $G$-orbit is a quasi-isometry onto its image, which will be sufficient to imply that $\overline{H}$ will be quasi-convex in $\overline{X}$ and $\overline{f}$ will be a loxodromic, WPD element transverse to $\overline{H}$.

Let $\overline{\mu_i}$ be the push-forward of $\mu_i$ to $\overline{G}$ for each $1\leq i\leq k$.  It is clear that each $\overline{\mu_i}$ is a permissible probability distribution on $\overline{G}$. If $R$ is a random subgroup of $G$ with respect to $\mu_1,\dots, \mu_k$, then its image $\overline{R}$ in $\overline{G}$ is a random subgroup of $\overline{G}$ with respect to $\overline{\mu_1},\dots,\overline{\mu_k}$. We have shown that $\langle\overline{H}, \overline{R}\rangle\cong\overline{H}\ast\overline{R}$, and this subgroup is quasi-convex in $\overline{X}$. Taking pre-images, we see that \begin{equation}\label{eqn:amalprod}\langle HE(G), RE(G)\rangle\cong HE(G)\ast_{E(G)}RE(G).\end{equation} Since $R$ is torsion-free, we have $R\cap E(G)=\{1\}$.  If we also assume that $H\cap E(G)=\{1\}$, then the subgroup of the amalgamated product (\ref{eqn:amalprod}) generated by $H$ and $R$ will be the free product of $H$ and $R$. Moreover, since the map $X\to\overline{X}$ is a $G$--equivariant quasi-isometry on $G$-orbits, $\langle H, R\rangle$ will be quasi-convex in $X$ and quasi-isometrically embedded in $X$ whenever $H$ is quasi-isometrically embedded in $X$.

\end{proof}

\begin{rem}\label{finnormsub}
The assumption that $H\cap E(G)=1$ is necessary to obtain a free product $H\ast R$, since for any $g\in G$ and any $k\in E(G)$, $k$ conjugated by $g^n$ is equal to $k$ for $n=|E(G)|!$. However, if this assumption is dropped, then it is clear from the proof that the subgroup $\langle H, R \rangle$ is still quasi-convex in $X$ and there is an isomorphism $\langle H, R, E(G)\rangle\cong HE(G)\ast_{E(G)} RE(G)$. 

\end{rem}

We end this section by proving Theorem \ref{main:eq}.

\maineq*

\begin{proof}
Let $W\in G\ast \mathbb F(x_1,\dots, x_k)$, which we identify with the normal form of $W$ consisting of a reduced word in the alphabet $G\cup \{x_1,\dots, x_k\}$ which does not contain consecutive letters from $G$. Let $g_1,\dots, g_l$ be the elements in $G$ which appear in $W$, and set $S=\{g_1,\dots,g_l\}$. Since $S$ is finite, the orbit $\pi(S)$ is bounded and hence quasi-convex in $X$.  Thus any loxodromic WPD element of $G$ will be transverse to $S$, and we can apply Theorem \ref{t:words} to conclude that  $W(w_1,\dots,w_k)\neq 1$ with probability going to 1 as $n\to\infty$.  
\end{proof}


\section{Geometric Separation}\label{sec:geomsep}
In this section, we prove Theorem \ref{thm:geomsep}.  We first recall the definition of geometric separation.
\begin{defn}\label{defn:geomsep}
Let $G$ be a group acting on a metric space $X$. The subgroup $H\leq G$ is  \emph{geometrically separated (in $X$)} if for all $\kappa\geq 0$, there exists $C>0$ such that for all $g\in G\setminus H$, 
\[
\diam(\pi(H)\cap \mc N_\kappa(g\pi(H)))<C.
\]
We refer to $C=C(\kappa)$ as the geometric separation constant.
\end{defn}

\geomsep*

\begin{proof}

Assume for now that $E(G)=\{1\}$.  Let $w_1,\dots, w_k$ by given by random walks of length $n$ with respect to $\mu_1,\dots,\mu_k$ respectively, and let $\gamma_i=[x_0, w_ix_0]$ and $R=\langle w_1,\dots, w_k\rangle$. Let $D_i$ denote the drift of $\mu_i$ and let $D=\min\{D_i\}$. Fix constants $\e>0$, $\e'>0$ chosen sufficiently small relative to $\e$, and $\e''>0$ chosen sufficiently small relative to $\e'$. Assume $\pi(H)$ is $\sigma$--quasi-convex in $X$. Assume that  $H_n:=\langle H, R\rangle\cong H\ast R$, which occurs with probability approaching one by Theorem \ref{thm:main1}. Let $M$ be the Morse constant for $(8, c)$--quasi-geodesics, where $c$ is given by applying Theorem \ref{t:words}. By remark \ref{rem:c}, we can assume that $c$ is a constant multiple of $\e'' Dn$ plus a constant multiple of $\delta$. By Lemma \ref{lem:Morsebd}, the same is true of $M$. In addition, we will assume that the following hold for each $1\leq i\leq k$.
\begin{enumerate}[(a)]
\item $\gamma_i$ has no $(\e' Dn, 6\delta+6M)$--self-match.
\item $\gamma_i$ has no $(\e' Dn, 6\delta+6M)$--match with $\gamma_j$ for any $j\neq i$.
\item $\gamma_i$ has no subpath of length $\e' Dn$  contained in the  $(6\delta+6M+\sigma)$--neighborhood of a translate of  $\pi(H)$.
\item $(1-\e)Dn\leq (1-\e)D_in\leq d(x_0,w_{i}x_0)\leq(1+\e)D_in$.
\end{enumerate}

Property (d) holds with probability going to 1 as $n\to\infty$ by Theorem \ref{thm:drift}. By choosing $\e''$ sufficiently small relative to $\e'$, it can be shown that (a)--(c) hold with probability going to 1 as $n\to\infty$ in the same way as in the proof of Theorem \ref{t:words}.


Let $L=\max\{\l(\gamma_i)\}$, and note that $L\geq (1-\e)Dn$ by (d). Our goal is to prove that $H_n$ is geometrically separated in $X$. Fix $\kappa\geq 0$, and let $C'$ be the constant such that if
\[
\diam(\pi(H)\cap \mc N_\kappa(g\pi(H)))>C', 
\]
then $g\in H$; the constant $C'$ exists as $H$ is geometrically separated in $X$. Now let $C=3\max\{L, C''\}+2(\kappa+2\delta)+2M$, where $C''$ will be specified later. Suppose there exists $g\in G$ such that \[\operatorname{diam}(\pi(H_n)\cap \mc{N}_\kappa(g\pi(H_n))\geq C.\]

Then there exist points $x_1,y_1\in \pi(H_n)$ with $d(x_1, y_1)\geq C$ and points $x_2,y_2\in \pi(g(H_n))$ with $d(x_1, x_2)\leq\kappa$ and $d(y_1, y_2)\leq\kappa$. By Theorem \ref{t:words}, there are $(8, c)$--quasi-geodesics $p_1$ from $x_1$ to $y_1$ and $p_2$ from $x_2$ to $y_2$ that are labeled by $H\cup\{w_1,\dots, w_k\}$.

The $(\kappa+2\delta)$--central segment of  the geodesic $[x_1,y_1]$ has length at least  $C-2(\kappa+2\delta)$ and is contained in the $2\delta$--neighborhood of $[x_2, y_2]$ by Lemma \ref{lem:centralseg}. There is a subpath $q_1$ of $p_1$ which is contained in the $M$--neighborhood of the $(\kappa+2\delta)$--central segment of $[x_1, y_1]$ with $\l(q_1)\geq C-2(\kappa+2\delta)-2M$. Hence $q_1\subseteq \mc{N}_{2\delta+2M}(p_2)$.

By assumption, $\l(q_1)\geq 3L$, hence there is a subpath $\eta$ of $q_1$ such that either $\eta$ is a translate by an element of $H_n$ of some $\gamma_i^{\pm 1}$, or $\eta$ is contained in a subpath labeled by an element of $H$ and $\l(\eta)\geq \frac13 \l(q_1)\geq \max\{L, C''\}$. In either case, we conclude that $\eta$ is a geodesic with $\l(\eta)\geq (1-\e)Dn$.

 Let $\beta$ be the subpath of $p_2$ such that $d(\eta_-, \beta_-)\leq 2\delta+2M$ and $d(\eta_+, \beta_+)\leq 2\delta+2M$. Hence $\l(\beta)\geq d(\beta_-, \beta_+)\geq d(\eta_-, \eta_+)-4\delta-4M=\l(\eta)-4\delta-4M\geq (1-\e)Dn-4\delta-4M$.  Note also the $d_{Hau}(\eta, \beta)\leq 4\delta+4M$.

We now choose a subpath $\beta'$ of $\beta$. If $\beta$ contains a translate of some $\gamma_j^{\pm 1}$, then we let $\beta'$ be this translate. Note that $\beta'=gh'\gamma_j$ for some $h'\in H_n$ in this case. If no such translate exists, then $\beta$ must have a subpath $\beta'$ such that $\l(\beta')\geq\frac13\l(\beta)$ and either $\beta'$ is contained in a translate by an element of $gH_n$ of some $\gamma_j^{\pm1}$ or $\beta'$ is contained in a subpath labeled by an element of $H$. In the first case we get $\l(\beta')\geq(1-\e)Dn$, and in the second $\l(\beta')\geq\frac13\l(\beta)\geq \frac{1}{3}((1-\e)Dn-4\delta-4M)$. 
Let $\eta'$ be the subpath of $\eta$ such that $d(\eta'_-, \beta'_-)\leq 4\delta+4M$ and $d(\eta'_+, \beta'_+)\leq 4\delta+4M$. Note that

\[
\l(\eta')\geq d(\beta'_-, \beta'_+)-8\delta-8M=\l(\beta')-8\delta-8M.
\]
 We also have $d_{Hau}(\eta', \beta')\leq6\delta+6M$. Recall that $M$ is a constant multiple of $\e'' Dn$ plus a constant multiple of $\delta$ and $\e'' Dn<\e Dn$. Hence we can choose $\e'$ such that for sufficiently large $n$, $\beta'$ and $\eta'$ both have length at least $\e'Dn$. We also choose $C''$ such that in the case where $\eta'$ and $\beta'$ are both contained in subpaths labeled by elements of $H$, we have $d(\eta'_-,\eta'_+)\geq C'$.

There are four possible cases, depending on the form of $\eta'$ and $\beta'$.

\noindent {\bf Case 1.} Suppose that $\eta'$ is contained in $h\gamma_i^{\pm 1}$ for some $h\in H_n$ and $\beta'$ is contained in $gh'\gamma_j$ for some $h'\in H_n$. The only way this does not give a contradiction with either (a) or (b) is if $i=j$ and $h^{-1}gh'=1$, hence $g\in H_n$ in this case.

\noindent{\bf Case 2.} Suppose that $\eta'$ is contained in $h\gamma_i^{\pm 1}$ for some $h\in H_n$ and $\beta'$ is contained in a subpath labeled by an element of $H$. Since $H$ is $\sigma$--quasi-convex, this implies that $\eta'$ is contained in the $(6\delta+6M+\sigma)$--neighborhood of a translate of $\pi(H)$. But this is a contradiction with (c).

\noindent{\bf Case 3.} Suppose $\eta'$ is contained in a subpath labeled by an element of $H$  and $\beta'$ is contained in a translate of some $\gamma_j^{\pm}$. In this case we get the same contradiction as in Case 2.

\noindent{\bf Case 4.} $\eta'$ and $\beta'$ are both contained in subpaths labeled by elements of $H$. Since $H$ is $\sigma$-quasi-convex, $\eta'$ is contained in the $\sigma$--neighborhood of $h\pi(H)$ and $\beta'$ is contained in the $\sigma$--neighborhood of $gh'\pi(H)$ for some $h, h'\in H_n$. By our choice of $C''$, we have $\diam(\pi(H)\cap\mc{N}_{6\delta+6M+2\sigma}h^{-1}gh'\pi(H))\geq C'$, which implies that $h^{-1}gh'\in H$ and hence $g\in H_n$.

Finally, we consider the case where $E(G)\neq 1$. We let $\overline{G}$, $\overline{H}$, $\overline{R}$, and $\overline{X}$ be as in the proof of Theorem \ref{thm:main1}. As in the proof of Theorem \ref{thm:main1}, we can apply the above proof to obtain that $\langle \overline{H}, \overline{G}\rangle$ is geometrically separated in $\overline{X}$ and then take pre-images to get that $\langle H, R, E(G)\rangle$ is geometrically separated in $G$. 
\end{proof}

We next discuss how Theorem \ref{thm:geomsep} applies to the theory of hyperbolically embedded subgroups introduced in \cite{DGO}. We will not need the full definition of a hyperbolically embedded subgroup here, since we will only make use of the following criteria:

\begin{thm}[{\cite[Theorem 4.2]{DGO}}]\label{thm:hypembed}
Suppose that $G$ acts on a hyperbolic space $X$ such that $H$ is quasi-isometrically embedded and geometrically separated in $X$. Then $H$ hyperbolically embeds in $G$.
\end{thm}

Since $H$ quasi-isometrically embeds into a hyperbolic metric space, $H$ itself must be a hyperbolic group. The converse of Theorem \ref{thm:hypembed} holds if one additionally assumes that $H$ is a proper, hyperbolic subgroup of $G$. Indeed, if $H$ is a proper hyperbolic subgroup of $G$ such that $H$ hyperbolically embeds in $G$, then $G$ has a (usually infinite) generating set $S$ such that the Cayley graph $\Cay(G, S)$ of $G$ with respect to $S$ is a hyperbolic metric space and $H$ is both quasi-isometrically embedded and geometrically separated in $\Cay(G, S)$. This follows from \cite[Theorem 6.4]{Sis12}, \cite[Lemma 3.1]{AMS}, and \cite[Lemma 3.2]{AMS}. Moreoever, \cite[Theorem 6.11]{DGO} shows that $G$ contains loxodromic WPD elements for the action of $G$ on $\Cay(G, S\cup H)$. Since the map $\Cay(G, S)\to\Cay(G, S\cup H)$ is 1-Lipschitz, these elements will also be loxodromic WPD elements for action of $G$ on $\Cay(G, S)$. Since they act loxodromically on $\Cay(G, S\cup H)$, these elements must be transverse to $H$ in $\Cay(G, S)$. This shows that Theorem \ref{thm:main1} and Theorem \ref{thm:geomsep} can be applied for any hyperbolic, hyperbolically embedded subgroup $H$ and any sequence $(\mu_i)$ of finitely supported probability distributions of full support.

Finally, we note if $G$, $H$, and $X$ are as in Theorem \ref{thm:hypembed} and, in addition, the action of $G$ on $X$ is cobounded, then the $S$ as above can be chosen such that $\Cay(G, S)$ is $G$--equivariently quasi-isometric to $X$ by \cite[Theorem 3.16]{H16} or \cite[Corollary 3.10]{AMS}.

Combining Theorems \ref{thm:main1} and \ref{thm:geomsep}, Remark \ref{finnormsub}, and the above discussion, we obtain:

\begin{cor}\label{cor:hyphypembed}
Suppose $H<G$ is hyperbolic and $H$ hyperbolically embeds in $G$. Let $(\mu_i)$ be a sequence of finitely supported probability distributions of full support. Then a random subgroup $R$ will satisfy $\langle H, R, E(G)\rangle\cong HE(G)\ast_{E(G)} RE(G)$ and $\langle H, R, E(G)\rangle$ hyperbolically embeds in $G$. 
\end{cor}
Note that when $H$ is infinite, $E(G)$ is a subgroup of $H$ by \cite[Theorem 6.14]{DGO}, and hence $\langle H, R, E(G)\rangle=\langle H, R\rangle$ in this case. When $E(G)=\{1\}$, we obtain that every hyperbolic, hyperbolically embedded subgroup of $G$ is a free factor in a larger hyperbolically embedded subgroup of $G$.

If $H$ hyperbolically embeds in $G$ but $H$ is not hyperbolic, then $G$ will still have a non-elementary, partially WPD action on a hyperbolic metric space $X$ with $H$ elliptic \cite{DGO}. In this case $H$ must be infinite, so $E(G)\leq H$ and hence $\langle H, R, E(G)\rangle=\langle H, R\rangle$. Theorem \ref{thm:main1} will then imply that for any sequence of finitely supported probability distributions of full support, a random subgroup $R$ will satisfy $\langle H, R\rangle\cong H\ast_{E(G)}RE(G)$. However we do not know the answer to the following:
\begin{ques}
Suppose $H$ hyperbolically embeds in $G$, but $H$ is not hyperbolic. If $R$ is a random subgroup of $G$, does $\langle H, R\rangle$ hyperbolically embed in $G$?
\end{ques}


\section{Stable subgroups}\label{sec:stablesubgp}

In   Section \ref{sec:App}, we will describe several situations in which Theorem \ref{thm:main1} applies. In order to apply this theorem we need to be able to find a loxodromic WPD element transverse to $H$.  In this section, we prove a general algebraic statement about stable subgroups of acylindrically hyperbolic groups which will be used in the next section in order to find the transverse elements. When $G$ is hyperbolic this statement is exactly equivalent to the existence of a transverse element, while in the other cases there is still  further work to be done.

When $Y$ is a subset of a geodesic metric space $X$, we say $Y$ is \emph{M--Morse} if for all $\lambda, c$, there exists $M=M(\lambda, c)$ such that any $(\lambda, c)$--quasi-geodesic in $X$ with endpoints in $Y$ is contained in $\mc N_M(Y)$. 

Let $G$ be generated by a finite set $S$. A finitely generated subgroup $H\leq G$ is called \emph{stable} if $H$ is quasi-isometrically embedded in $\Cay(G, S)$ and for all $\lambda, c$, there exists $R=R(\lambda, c)$ such that if $p$ and $q$ are $(\lambda, c)$--quasi-geodesics in $\Cay(G, S)$ with equal endpoints in $H$, then $p\subseteq \mc{N}_R(q)$. This definition is due to Durham--Tayor \cite{DT} and  is equivalent to requiring that $H$ is a hyperbolic group and $H$ is a Morse subset of $\Cay(G, S)$ \cite{CorHum}. It is easy to see that this notion is independent of the choice of finite generating sets for $G$ and $H$.

The following is the main theorem of this section. 
\begin{thm} \label{thm:stablesubgroup}
Let $H$ be an infinite index stable subgroup of a finitely generated group $G$. Suppose $G$ has a non-elementary WPD action on a hyperbolic space $X$. Then there exists a loxodromic element $f$ such that  $H^g\cap\langle f\rangle=\{1\}$ for all $g\in G$. 
\end{thm}

This is a generalization of \cite[Proposition 1]{Minasyan2}, where it is proved in the case that $G$ is hyperbolic and $X=\Cay(G, S)$. As in \cite[Proposition 1]{Minasyan2}, the element $f$ can be chosen to belong to any subgroup $K$ such that the intersection of $K$ and any conjugate of $H$ is an infinite index subgroup of $K$, as long as $K$ contains at least one loxodromic WPD element (this is automatic when $K$ is an infinite subgroup of a hyperbolic group). We  note that if $H$ is stable in $G$ and $G$ is not hyperbolic, then $H$ is necessarily infinite index.

 As in the case of quasi-convex subgroups of hyperbolic groups, we will use the fact that stable subgroups have finite width and finite height \cite[Theorem 1.1]{AMST}. Recall that $H$ has \emph{width $\leq n$ in $G$} if for any distinct cosets $g_1H,\ldots,g_nH$, there exist $1\leq i<j\leq n$ such that $H^{g_i}\cap H^{g_j}$ is finite. $H$ has \emph{height $\leq n$} if for any distinct cosets $g_1H,\ldots,g_nH$, the intersection $H^{g_1}\cap\ldots\cap H^{g_n}$ is finite.


\begin{rem}
The proof of Theorem \ref{thm:stablesubgroup} will hold for any infinite index subgroup $H$ of finite height and finite width such that $H$ is a Morse subset of $\Cay(G, S)$.
\end{rem}

For the proof of Theorem \ref{thm:stablesubgroup} we will more or less follow the same steps as the proof of \cite[Proposition~1]{Minasyan2}. Throughout the rest of this section, we fix a group $G$ with a non-elementary action on a hyperbolic space $X$. For now we will only assume that the action of $G$ on $X$ contains WPD elements, though in the final step of the proof we will need the assumptions that the action is WPD, that is, \emph{all} loxodromic elements are WPD.

 We also fix a finite generating set $S$ for $G$ and let $d_S$ denote the corresponding word metric on $G$. We will use $d_X$ to denote the metric on $X$.

\begin{lem}\label{lem:oneconj}
Let $H$ be an infinite index stable subgroup of $G$. Then there exists a loxodromic WPD element $f$ such that $H\cap\langle f\rangle=\{1\}$.
\end{lem}

\begin{proof}
Since any power of a loxodromic WPD element is again a loxodromic WPD element, if $H$ contains no such elements then the statement is obvious. Suppose now that $H$ does contain a loxodromic WPD element $h$. Since $H$ has infinite index and finite width, there exists some $g\in G$ such that $H^g\cap H$ is finite. Hence $\langle g^{-1}hg\rangle\cap H=\{1\}$, so we can set $f=g^{-1}hg$.
\end{proof}

\begin{lem}\label{lem:productsoflox}
Let $y_1,\ldots,y_s$, be loxodromic WPD elements of $G$ such that $E(y_i)\neq E(y_j)$ for $i\neq j$. Then there exist $\lambda, c$, and $N$ such that for any $i_1,\ldots,i_t\in\{1,\ldots,s\}$ with each $i_k\neq i_{k+1} \mod t$ and $m_1,\ldots, m_s\in\Z$ with each $m_i\geq N$, the element $z=y_{i_1}^{m_1}\ldots y_{i_t}^{m_t}$ is loxodromic. Moreover, if $W_i$ is a shortest word in $S$ representing $y_i$, then any path in $\Cay(G, S)$ labeled by 
\begin{equation}\label{eq:qgwords}
W_{i_i}^{m_1}\ldots W_{i_t}^{m_t}
\end{equation}
is a $(\lambda, c)$-quasi-geodesic. 
\end{lem}

\begin{proof}

Fix a basepoint $x_0\in X$, and let $\alpha$ be the path from $x_0$ to $zx_0$ obtained by concatenating geodesics $[p, y_{i_1}p]\cdot[y_{i_1}p, y_{i_1}^2p]\cdot \ldots \cdot [y_{i_1}^{m_1}\ldots y_{i_t}^{m_t-1}p, y_{i_1}^{m_1}\ldots y_{i_t}^{m_t}p]$. In other words, $\alpha$ is the path in $X$ based at $x_0$ and labeled by $W_{i_i}^{m_1}\cdots W_{i_t}^{m_t}$. 

We claim that there exist constants $\nu$ and $d$ depending only on $y_1,\dots, y_s$ (in particular, independent of $m_1,\dots,m_t$) such that $\alpha$ is a $(\nu, d)$--quasi-geodesic. The proof is essentially the same as \cite[Lemma 2.3]{Ols93}, where this is proved in the case that $G$ is hyperbolic and $X=\Cay(G, S)$. Formally the fact that $X$ is a Cayley graph with respect to a finite generating set is used only once in the proof of \cite[Lemma 2.3]{Ols93} in the form of references to previous lemmas (\cite[Lemma 2.1]{Ols93} and \cite[Lemma 2.2]{Ols93}) that are used to show that for infinite order elements $g$ and $h$ of a hyperbolic group with $E(g)\neq E(h)$, if $d_S(g^m, h^{n})<C$ then there is a bound on $m$ depending on $g$, $h$ and $C$. In the setting where $g$ and $h$ are loxodromic WPD elements we can obtain the same conclusion (with $d_S(g^{m}, h^{n})$ replaced by $d_X(g^mx_0, h^{n}x_0)$) by \cite[Proposition 6]{BestvinaFujiwara}. The rest of the proof of \cite[Lemma 2.3]{Ols93} works in our situation with no essential changes.

Note that for any $k\geq 1$, the word $(W_{i_i}^{m_1}\cdots W_{i_t}^{m_t})^k$ is in the same form as (\ref{eq:qgwords}). Hence for any $k$, the concatenation of paths $\alpha\cdot z\alpha\cdot\ldots\cdot z^k\alpha$ is a $(\nu, d)$--quasi-geodesic from $x_0$ to $z^kx_0$. It follows that $z$ acts loxodromically on $X$.

Next, we prove that any path in $\Cay(G, S)$ labeled by $W_{i_i}^{m_1}\cdots W_{i_t}^{m_t}$ is a quasi-geodesic. In order to simplify notation we will assume that $y_1,\dots,y_s\in S$, hence each $W_{i_j}$ consists of a single letter. Since the constants $(\lambda, c)$ are allowed to depend on $y_1,..,y_s$, there is no loss of generality here.

Consider the $(\nu, d)$--quasi-geodesic path $\alpha$ in $X$ constructed above. Let $D_1=\max_{g\in S}d_X(x_0, gx_0)$ and $D_2=\min_{1\leq i\leq s}d(x_0, y_ix_0)$. Note that $D_2>0$ since each $y_i$ has no fixed points. Also note that $\l(\alpha)\geq D_2(\sum m_i)$ and $d_X(x_0, zx_0)\leq D_1 |z|_S$. Since $\alpha$ is a $(\nu, d)$--quasi-geodesic, we have $D_2(\sum m_i) \leq\l(\alpha)\leq \nu d_X(x_0, zx_0)+d\leq \nu D_1 |z|_S+d$. Since each $W_i$ consists of a single letter,
\[
\|W_{i_i}^{m_1}\cdots W_{i_t}^{m_t}\|=\sum m_i\leq\frac{\nu D_1}{D_2} |z|_S+\frac{d}{D_2}.
\]

Finally, any subword of  $W_{i_i}^{m_1}\cdots W_{i_t}^{m_t}$ will be a word of the form $U_1WU_2$, where each $U_i$ has length at most $N$ and $W$ is a word in the same form as (\ref{eq:qgwords}). Hence after modifying the additive constant, the same proof applies to subwords of $W_{i_i}^{m_1}\cdots W_{i_t}^{m_t}$.
\end{proof}

\begin{lem}\label{lem:ftlox}
Let $H\leq G$ and let $f\in G$ be an infinite order element. Suppose there exists a constant $K\geq 0$ and a sequence $(n_i)$ such that $d_S(f^{n_i}, H)\leq K$. Then $\langle f\rangle\cap H\neq \{1\}$.
\end{lem}

\begin{proof}
Let $\Delta=\{g\in \langle f\rangle H\;|\; |g|_S\leq K\}$. For each $g\in\Delta$, fix $x_g\in\langle f\rangle$ and $y_g\in H$ such that $x_g^{-1}y_g=g$. Let $\Omega=\{x_g\;|\; g\in\Delta\}$. Note that $\Delta$ is finite and hence $\Omega$ is finite.

By assumption, there exists $h_i\in H$ such that $|f^{-n_i}h_i|_S\leq K$. Hence $f^{-n_i}h_i\in \Delta$, so $f^{-n_i}h_i=x_i^{-1}y_i$ for some $x_i\in \Omega$ and $y_i\in H$. Hence $x_if^{-n_i}=y_ih_i^{-1}\in\langle f\rangle\cap H$. Since $\Omega$ is finite, there exists $n_i$ such that $f^{n_i}\notin\Omega$, hence $x_if^{-n_i}\neq 1$.
\end{proof}

\begin{lem}\label{lem:s-conj}
Let $H$ be an infinite index stable subgroup of $G$. Then for any $g_1,\ldots,g_s\in G$, there exists a loxodromic element $f$ such that 
\[
\langle f\rangle\cap(H^{g_1}\cup \cdots\cup H^{g_s})=\{1\}.
\]
\end{lem}

\begin{proof}
By Lemma \ref{lem:oneconj}, we can choose loxodromic WPD elements $x_i$ such that $\langle x_i\rangle\cap H^{g_i}=\{1\}$.  We will assume that  $E(x_i)\neq E(x_j)$ for $i\neq j$; otherwise we can remove $x_j$ from our list and perform the following construction without it.

Define $$\alpha_{ij}=\begin{cases} \min\{m\;|\; x_i^m\in H^{g_j}\} & \text{if such an $m$ exists}\\  1&\text{if no such $m$ exists.} \end{cases}$$

Let $\alpha_i=\operatorname{lcm}(\{\alpha_{i1},\dots,\alpha_{is}\})$, and let $y_i=x_i^{\alpha_i}$. In particular, $y_i$ is chosen such that for all $j$, either $\langle y_i\rangle \cap H^{g_j}=\{1\}$ or $y_i\in H^{g_j}$. Note that the first possibility occurs when $i=j$.

Now, let $z_n=y_1^ny_2^n\cdots y_s^n$. For all sufficiently large $n$, the element $z_n$ will act loxodromically on $X$ by Lemma \ref{lem:productsoflox}. Assume for the sake of contradiction that for all $n$, there exists $l_n$ such that $z_n^{l_n}\in H^{g_1}\cup \cdots\cup H^{g_s}$. Hence for some index $j$, there are infinitely many $n$ such that $z_n^{l_n}\in H^{g_j}$. Let $W_i$ be a shortest word in $S$ representing $y_i$, and consider a path $\alpha_n$ in $\Cay(G, S)$ labeled by $(W_1^n\cdots W_s^n)^{l_n}$. By Lemma \ref{lem:productsoflox}, $\alpha_n$ is a $(\lambda, c)$--quasi-geodesic where  $\lambda$ and $c$ are independent of $n$. Since $H^{g_j}$ is a Morse subset of $\Cay(G, S)$, there exists $M=M(\lambda, c)$ such that $\alpha_n\subseteq\mc{N}_M( H^{g_j})$ for infinitely many $n$. Choose $t=\min\{ i\;|\; y_i\notin H^{g_j}\}$, and note that $1\leq t\leq j$. Now $d_S(y_1^n\cdots y_t^n, H^{g_j})\leq M$ for infinitely many $n$. Since $y_1,\ldots, y_{t-1}\in H^{g_j}$ by our choice of $t$, there exists some $a_n$ such that $y_t^na_n\in H^{g_j}$ and $|a_n|_S\leq M$ for infinitely many $n$. Then $\langle y_t\rangle \cap H^{g_j}\neq\{1\}$ by Lemma \ref{lem:ftlox}, and hence $y_t\in H^{g_j}$ by construction. But this contradicts our choice of $y_t$.
\end{proof}

Finally, in order to complete the proof of Theorem \ref{thm:stablesubgroup} we will assume that the action of $G$ on $X$ is WPD. This implies that the loxodromic elements produced by Lemma \ref{lem:s-conj} are in fact WPD elements.
\begin{proof}[Proof of Theorem \ref{thm:stablesubgroup}]

Assume for the sake of contradiction that every loxodromic element $x$ has a power contained in a conjugate of $H$. Let $h\in G$ be a fixed loxodromic element, and let $l_0\geq 1$ and $g_0\in G$ be such that $h^{l_0}\in H^{g_0}$. Let $y$ be any loxodromic element. We first assume that $E(y)\neq E(h)$. By Lemma \ref{lem:productsoflox}, there exists $N$, $\lambda$, and $c$ such that  $y^nh^n$ is a loxodromic element for all $n\geq N$. Moreover, if $U$ and $W$ are shortest words in $S$ representing $h$ and $y$ respectively, then for any $k\geq 1$ any path in $\Cay(G, S)$ labeled by $(U^nW^n)^k$ is a $(\lambda, c)$--quasi-geodesic.  By assumption, for all $n$ there exists $l_n\geq 1$, $g_n\in G$ such that $g_n(y^nh^n)^{l_n}g_n^{-1}=h_n\in H$. Consider the path from $1$ to $h_n^k$ which is formed by concatenating paths labeled by $g_n$, $(y^nh^n)^{kl_n}$, and $g_n^{-1}$, where $k$ is chosen sufficiently large (see Figure \ref{fig:quad}). Let $p_n$ be the path labeled by $(y^nh^n)^{kl_n}$. Let $q_1$ be a shortest path from $1$ to $p_n$ and $q_2$ a shortest path from $h_n^k$ to $p_n$. Let $p_n^{\prime}$ be the subpath of $p_n$ between the endpoints of $q_1$ and $q_2$. Then for $k$ sufficiently large, $q_1p_n^\prime q_2^{-1}$ is a $(3\lambda, c)$--quasi-geodesic.  

 Since $H$ is a Morse subset of $\Cay(G, S)$, there exists $M=M(3\lambda, c)$ such that $q_1p_n^\prime q_2\subseteq \mc{N}_M(H)$. For $k$ sufficiently large, $p_n^\prime$ will contain a subpath labeled by $W^nU^n$, so there exist $u_n,v_n,z_n\in G$, each of word length $\leq M$, such that $u_ny^nv_n^{-1}\in H$ and $v_nh^nz_n^{-1}\in H$. Since there are only finitely many such words, for some $i\neq j$, we have $u_i=u_j$, $v_i=v_j$, and $z_i=z_j$. Hence $v_iy^{j-i}v_i^{-1}\in H$ and $v_i h^{j-i}v_i^{-1}\in H$.
 
 \begin{figure}
\def\svgwidth{4in}  
  \centering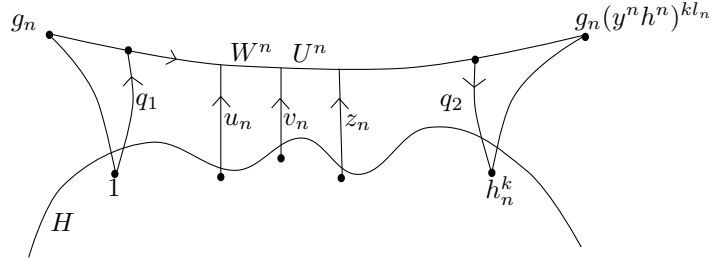 \\
	\caption{Proof of Theorem \ref{thm:stablesubgroup}.} 
	\label{fig:quad}
\end{figure}

That is, for all loxodromic elements $y$, there exists $a=a(y)$ and $l=l(y)\neq 0$ such that $ay^la^{-1}\in H$ and $ah^la^{-1}\in H$. Note that if $E(y)=E(h)$, then there exists $m$ such that $y^m\in\langle h\rangle$, and we can take $l=ml_0$ and $a=g_0$. Now if $y_1,\dots,y_n$ are loxodromic elements and $a_i=a(y_i)$ and $l_i=l(y_i)$, then $y_i^{l_i}\in H^{a_i}$. Moreover, $h^{l_1l_2\cdots l_n}\in H^{a_1^{-1}}\cap\cdots \cap H^{a_n^{-1}}$, hence this intersection is infinite. However, $H$ has finite height, so there are only finitely many distinct $a_i^{-1}$ (and hence distinct $a_i$)  that occur. But  this gives a finite union of conjugates of $H$ which contains a power of any loxodromic element, contradicting Lemma \ref{lem:s-conj}.
\end{proof}

Let $S$ be a non-exceptional surface and $\MCG(S)$ the mapping class group of $S$. A subgroup of $\MCG(S)$ is \emph{convex cocompact} if and only if it is stable \cite{DT}.  Also, an element of $\MCG(S)$ acts as a loxodromic WPD isometry of the curve graph of $S$ if and only if it is pseudo-Anosov. Hence we obtain the following corollary.

\begin{cor}\label{cor:MCG}
Let $S$ be a non-exceptional surface and let $H\leq \MCG(S)$ be a convex cocompact subgroup. Then there exists a pseudo-Anosov $f\in \MCG(S)$ such that $\langle f\rangle \cap H^g=\{1\}$ for all $g \in \MCG(S)$.
\end{cor}

Next we show that a result analogous to Theorem \ref{thm:stablesubgroup} holds for relatively quasi-convex subgroups of relatively hyperbolic groups. Let $G$ be a group hyperbolic relative to a collection of proper subgroups $\{P_1,\ldots, P_n\}$, called \emph{peripheral subgroups}. In this case we refer to $G$ as a relatively hyperbolic group. Let $S$ be a finite relative generating set for $G$, i.e., $G$ is generated by $S\cup \mc{P}$, where $\mc{P}= P_1\cup\cdots\cup P_n$. Then the metric space $\Cay(G, S\cup \mc{P})$ is hyperbolic.  A subgroup $H\leq G$ is \emph{$\sigma$--relatively quasi-convex} if for any geodesic $q$ in $\Cay(G, S\cup \mc{P})$ with endpoints in $H$ and any vertex $x\in q$, we have $d_S(x, H)\leq \sigma$. Here $d_S(g, h)$ is defined to be the length of the shortest word in $S$ representing $g^{-1}h$ or $\infty$ if no such word exists. This is equivalent to several other definition of relative quasi-convexity, see \cite{Hruska, ManMP}. See also \cite{MP} for related work on combination theorems for relatively quasi-convex subgroups.  Generally speaking, relatively quasi-convex subgroups do not have to be stable in $G$. For example, each peripheral subgroup $P_i$ is relatively quasi-convex, but $P_i$ will be stable in $G$ if and only if it is hyperbolic. In fact, if each peripheral subgroup is one-ended with linear divergence then a relatively quasi-convex subgroup of $G$ is stable if and only if it has a finite intersection with each conjugate of each $P_i$ by \cite[Theorem 5.4]{AouDurTay} and \cite[Theorem 4.13]{Osin:RelHyp}. See \cite{Tran} for more on stable subgroups of relatively hyperbolic groups.

An element of a relatively hyperbolic group $G$ is \emph{elliptic} if it has finite order, \emph{parabolic} if it has infinite order and is conjugate into some peripheral subgroup, and \emph{hyperbolic} otherwise.

\begin{thm}\label{thm:relqc}
Let $H$ be an infinite index, relatively quasi-convex subgroup of a relatively hyperbolic group $G$. Then there exists a hyperbolic element $f\in G$ such that  $\langle f\rangle \cap H^g=\{1\}$ for all $g \in G$.
\end{thm}

The proof of Theorem \ref{thm:relqc} is essentially the same as the proof of Theorem \ref{thm:stablesubgroup} with only minor modifications. We will explain the necessary modifications instead of repeating the entire argument.

Fix $G$ hyperbolic relative to $\{P_1,\ldots, P_n\}$ and let $S$ and $\mc{P}$ be defined as above. First, note that the analogue of Lemma \ref{lem:oneconj}, that is the fact that there exists a hyperbolic element $f\in G$ such that $\langle f\rangle\cap H=\{1\}$ for any infinite index relatively quasi-convex subgroup $H$, follows from \cite[Lemma 8.4]{HruWis}. We can now repeat the rest of the proof of Theorem \ref{thm:stablesubgroup} inside the space $\Cay(G, S\cup \mc{P})$ instead of $\Cay(G, S)$. There are two places where the fact that the metric on $\Cay(G, S)$ is locally finite is used: once in the proof of Lemma \ref{lem:s-conj} and once in the proof of Theorem \ref{thm:stablesubgroup}. In both cases, the corresponding bound on distances comes from the fact that $H$ is a Morse subset of $\Cay(G, S)$. In the relatively quasi-convex setting, we instead use that the definition of relative quasi-convexity gives a bound on these distances with respect to metric $d_S$. This metric is again locally finite, and hence the same conclusions hold in this setting. Finally, in the last step of the proof of Theorem \ref{thm:stablesubgroup} we use the fact that stable subgroups have finite height. In the relatively quasi-convex setting, we instead appeal to \cite[Theorem 1.4]{HruWis}. Note that the definition of finite height used in \cite{HruWis} is slightly different then the one given above as it has been adapted to the relative setting, but it still applies in the same way and gives the same conclusion as in the proof of Theorem \ref{thm:stablesubgroup}.

\section{Applications}\label{sec:App}
We are now ready to give several applications of our main theorems, including the proofs of Corollaries \ref{cor:elliptic} and \ref{cor:relqcx} and Theorems \ref{thm:hypqc} and \ref{thm:MCGCCC}.

\subsection{Hyperbolic groups}
Let $G$ be a non-elementary hyperbolic group generated by a finite subset $S$.  We will consider the action of $G$ on $\Cay(G,S)$. This action is proper and cocompact, hence acylindrical. A subgroup is elliptic for this action if and only if it is finite, hence Corollary \ref{cor:mainell} implies Corollary \ref{cor:elliptic}(1).

 It is well-known that a subgroup $H\leq G$ is quasi-convex in $\Cay(G, S)$ if and only if it is quasi-isometrically embedded in $\Cay(G,S)$.  Recall also that a subgroup $H$ is \emph{malnormal} if for all $g\in G\setminus H$, we have $H\cap H^g=\{1\}$, and \emph{almost malnormal} if for all $g\in G\setminus H$, the intersection $H\cap H^g$ is finite. The following lemma follows from results in \cite{DGO}, in particular \cite[Example 2.6, Proposition 2.10, \& Theorem 2.7]{DGO}.

\begin{lem}\label{lem:almal}
Let $H$ be a quasi-convex subgroup of a non-elementary hyperbolic group. Then $H$ hyperbolically embeds in $G$ if and only if $H$ is almost malnormal.
\end{lem}

The following theorem clearly implies Theorem \ref{thm:hypqc}.

\begin{thm} \label{thm:hypgps}
Let $G$ be a non-elementary hyperbolic group and let $(\mu_i)$ be a sequence of permissible probability distributions on $G$. Let $H$ be an infinite index quasi-convex subgroup of $G$. Then a random subgroup $R$ will satisfy $\langle H, R, E(G)\rangle\cong HE(G)\ast_{E(G)} RE(G)$ and $\langle H, R\rangle$ will be an infinite index quasi-convex subgroup of $G$. If, in addition, $H$ is almost malnormal in $G$, then $\langle H, R\rangle$ will also be almost malnormal in $G$.
\end{thm}

\begin{proof}

We will apply Theorem \ref{thm:main1} and Remark \ref{finnormsub} with $X=\Cay(G, S)$. For this action, an element $g\in G$ is loxodromic if and only if it is infinite order. By Theorem \ref{thm:stablesubgroup} (which, in this case, is exactly \cite[Proposition~1]{Minasyan2}), there is an infinite order element $f\in G$ such that $H^g\cap\langle f\rangle=\{1\}$ for all $g\in G$. This implies that $f$ is transverse to $H$ in $X$. The first statement of the theorem now follows by Theorem \ref{thm:main1} and the second statement by Corollary \ref{cor:hyphypembed} and Lemma \ref{lem:almal}. 

\end{proof}

\subsection{Relatively hyperbolic groups} \label{sec:relhyp}
 We continue using the notation and terminology from Section \ref{sec:stablesubgp} and refer to \cite{Hruska, Osin:RelHyp} for definitions and background on relatively hyperbolic groups. 
 The action of $G$ on $\Cay(G, S\cup\mc{P})$ is acylindrical \cite{Osin:RelHyp, OsinAH}.  Subgroups which are elliptic for this action are precisely those which are either finite or  conjugate into some peripheral subgroup, hence Corollary \ref{cor:mainell}  implies Corollary \ref{cor:elliptic}(2).


\begin{thm} \label{thm:relhypgps}
Let $H$ be an infinite-index, relatively quasi-convex subgroup of a relatively hyperbolic group $G$ such that $H\cap E(G)=\{1\}$. Let $(\mu_i)$ be a sequence of permissible probability distributions on $G$. Then a random subgroup $R$ satisfies $\langle H, R\rangle\cong H\ast R$ and $\langle H, R\rangle$ is relatively quasi-convex in $G$.
\end{thm}

 Recall that an infinite order element $f\in G$ is \emph{hyperbolic} if $f$ is infinite order and is not conjugate into $P_i$ for any $i$.
Before proving Theorem \ref{thm:relhypgps}, we show that we can always find a hyperbolic element $f\in G$ which is transverse to $H$.
\begin{prop} \label{lem:transverse}
Let $G$ be a relatively hyperbolic group and $H\leq G$ an infinite-index $\sigma$--relatively quasi-convex subgroup.  Then there exists a hyperbolic element $f\in G$ such that $f$ is transverse to $H$.
\end{prop}

\begin{proof} 

By Theorem \ref{thm:relqc} there exists a hyperbolic $f\in G$ such that $\langle f\rangle\cap H^g=\{1\}$ for all $g\in G$.  Let $\alpha_f$ be a $(\lambda,c)$--quasi-geodesic axis for $f$ in $\Cay(G,S\cup\mathcal P)$, which is formed as the concatenation of geodesic segments $f^i[1,f]$.  Let $M$ be the Morse constant for $(\lambda,c)$--quasi-geodesics in $\Cay(G, S\cup\mc{P})$.  

We recall some terminology from \cite[Section~3]{Osin:RelHyp}. A \emph{$P_i$--component} of a path in $\Cay(G, S\cup\mc{P})$ is a maximal subpath in which all edges are labeled by elements of $P_i$. Two $P_i$--components of path are \emph{connected} if there is an edge labeled by an element of $P_i$ between a vertex of one component and a vertex of the other component. A $P_i$--component is  \emph{isolated} if it is not connected to any other $P_i$--component of the path.  A path $p$ in $\Cay(G,S\cup\mathcal P)$ is called a \emph{path without backtracking} if for any $i=1,\dots, n$, every $P_i$--component of $p$ is isolated.  Let $v$ be a vertex of some $P_i$--component $s$ of $p$.  If $v\neq s_-$ and $v\neq s_+$, then $v$ is an \emph{inner vertex} of $s$.  A vertex $u$ of $p$ is called \emph{non-phase} if it is an inner vertex of some $P_i$--component; it is called \emph{phase} otherwise.  Two paths $p,q$ in $\Cay(G,S\cup\mathcal P)$ are called \emph{$k$--similar} if there exists a constant $k$ such that $\max\{d_S(p_-,q_-),d_S(p_+,q_+)\}\leq k$.\\

\begin{claim}
After possibly replacing $f$ by a power, there is a $(\lambda,c)$--quasi-geodesic $\alpha_f'$ in $\Cay(G,S\cup\mathcal P)$ without backtracking and a constant $\nu$ depending on $f$ such that $d_S(f^i,\alpha_f')\leq \nu$ for every $i$.
\end{claim}

\begin{proof}[Proof of claim]
If $\alpha_f$ has no backtracking, we are done, so assume this is not the case.  Since  $f^i[1,f]$ is a geodesic for each $i$, each $P_j$--component is isolated in that geodesic.  Thus any backtracking must come from concatenating multiple geodesics. Since the geodesic $\alpha_f$ is formed by concatenating the images of $[1,f]$ under the action of $\langle f\rangle$, it suffices to consider the $P_j$--components of $[1,f]$.

The axis $\alpha_f$ is a $(\lambda,c)$--quasi-geodesic, hence there is a constant $I$ depending on $\lambda,c$ and $f$ such that no $P_j$--component of $[1,f]$ can connect to an $P_j$--component of $[f^i,f^{i+1}]$ for any $i> I$.  To see this, let $u, v$ be $P_j$--components of $[1,f]$ and $[f^i,f^{i+1}]$, respectively.  Then $d_{S\cup\mc P}(u_-,v_+)=1$, which implies that the subpath of $\alpha_f$ from $u_-$ to $v_+$ can be no longer than $\lambda+c$, providing the desired $I$.  By passing to a power of $f$ if necessary, we may assume that $I=1$.

Thus it remains to consider the connected $P_j$--components of $[1,f]\cdot [f,f^2]$.  Let $u,w$ be connected $P_j$--components of $[1,f]$ and $[f,f^2]$, respectively, such that for all $k=1,\dots, n$, no $P_k$--component of $[1,u_-]$ connects to an $P_k$--component of $[v_+,f^2]$.  Let $[u_-,v_+]$ be an edge labeled by an element of $H_j$, and let $\alpha'_f$ be the path formed from $\alpha_f$ by replacing $[f^{i-1}u_-,f^i]\cdot[f^i,f^{i-1}v_+]$ with the edge $[f^{i-1}u_-,f^{i-1}v_+]$ for each $i$.  By construction, $\alpha_f'$ is a $(\lambda,c)$--quasi-geodesic (in fact, it is a quasi-geodesic with better constants, but this is sufficient for the proof) without backtracking.  Let $\nu=\max\{d_S(f,u_-),d_S(f,v_+)\}$.  Then since $G$ acts by isometries, $d_S(f^i,\alpha'_f)\leq \nu$, completing the proof of the claim.
\end{proof}

Note that in the proof of the claim, the vertices $u_-$,$v_+$ are phase vertices by construction, and thus all vertices $f^iu_-$ and $f^iv_+$ are also phase vertices.

To show that $f$ is transverse to $H$ in $\Cay(G, \mc{P}\cup S)$, it suffices to show that for all $K$, there exists a constant $B$ depending only on $f, H,$ $K$, and the hyperbolicity constant $\delta$ of $\Cay(G,S\cup\mathcal P)$ such that $\diam(\alpha_f' \cap \mathcal N_K(gH))\leq B$ for all $g\in G$.  Suppose towards a contradiction that this is not the case.  Then there exists $K$ such that for any $B$, there exists $g\in G$ with $\diam(\alpha_f' \cap \mathcal N_K(gH))> B$. Let $x,y\in \alpha'_f$ be the first and last vertices contained in the $K$--neighborhood of $gH$, respectively, and let $x',y'\in gH$ be such that $d_{S\cup\mc P}(x',x)\leq K$ and $d_{S\cup\mc P}(y',y)\leq K$.  Let $p$ be a geodesic in $\Cay(G,S\cup\mathcal P)$ from $x'$ to $y'$.  Since $\Cay(G,S\cup\mathcal P)$ is $\delta$-hyperbolic, when $B$ is sufficiently large compared to $K$ there is a subpath $p'$ of $p$ which is contained in the $(2\delta+M)$--neighborhood of $\alpha_f'$.  Let $a',b'\in p$ be the first and last vertices of $p'$, and let $a,b\in \alpha_f'$ be the closest vertices to $a',b'$, respectively, so that $d_{S\cup\mc P}(a',a)\leq 2\delta+M$ and $d_{S\cup\mc P}(b',b)\leq 2\delta+M$.  Let $q$ be the subpath of $\alpha_f'$ from $a$ to $b$, and let $q'=[a',a]\cdot q\cdot [b,b']$.  By choosing $B$ sufficiently large, we can ensure that the length of $q$ is sufficiently long in comparison to $2\delta+M$ and thus ensure that $q'$ is a quasi-geodesic whose constants do not depend on $K$.  See Figure \ref{fig:relhyp}.

Moreover, $q'$ is the concatenation of two geodesics and a path without backtracking, and so any backtracking in $q'$ must come from the concatenation.  By choosing $B$ sufficiently large, we can ensure that no $P_i$--component of $[a,a']$ is connected to an $P_i$--component of $[b,b']$.  Further, since the geodesic $[a,a']$ is connecting $a$ to its nearest point projection onto $\alpha_f'$, no $P_i$--component of $[a,a']$ can connect to a distinct $P_i$--component of $q$, and similarly for $P_i$--components of $[b,b']$.  Therefore, $q'$ is a path without backtracking.  

 \begin{figure}
\def\svgwidth{4in}  
  \centering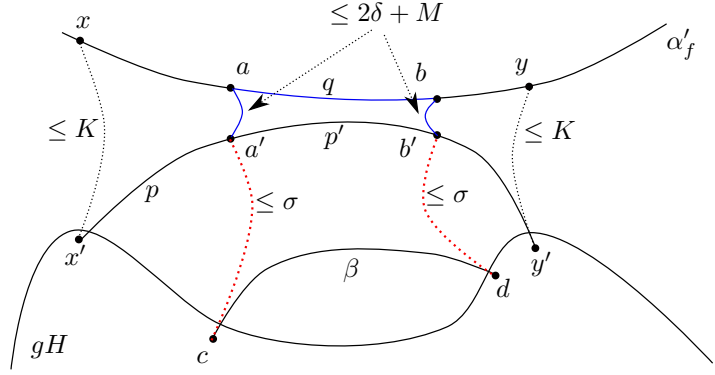 \\
	\caption{Proof of Proposition \ref{lem:transverse}.  All distances are measured with respect to $S\cup \mc P$ except for the red paths, for which the distances are measures with respect to $S$.  The concatenation of the blue paths is $q'$.} 
	\label{fig:relhyp}
\end{figure}

Since $a',b'\in p$ and $gH$ is $\sigma$--relatively quasi-convex, there are vertices $c,d\in gH$ with $d_S(c,a')\leq \sigma$ and $d_S(d,b')\leq \sigma$.  Let $\beta$ be a geodesic in $\Cay(G,S\cup\mathcal P)$ from $c$ to $d$.  Then $\beta$ and $q'$ are $\sigma$--similar paths in $\Cay(G,S\cup\mathcal P)$ without backtracking.  Note that for every vertex $u$ on $\beta$, we have $d_S(u,gH)\leq \sigma$.    Therefore, by \cite[Proposition~3.15]{Osin:RelHyp}, there is a constant $\e'=\e'(\lambda, c, \delta)$ such that for any phase vertex $w$ of $q'$, there exists a vertex $v$ of $\beta$ such that $d_S(w,v)\leq \e'$.

Since the vertices $f^iu_-$ in the proof of the claim are phase vertices for all $i$,  by again choosing $B$ large enough, there are arbitrarily many $i$ and vertices $\{gh_i\}$ in $gH$ such that $d_S(f^iu_-,gh_i)\leq \e'+\sigma$.  Since $d_S(f^iu_-,f^i)\leq \nu$, we have $d_S(f^i,gh_i)\leq \e'+\sigma+\nu$ and hence $d_S(1,f^{-i}gh_i)\leq \e'+\sigma+\nu$.  Since $\Cay(G,S)$ is a proper space, for large enough $B$ we must have some $i\neq j$ such that $f^{-i}gh_i=f^{-j}gh_j$, and consequently $f^{j-i}=gh_jh_i^{-1}g^{-1}\in H^g$, which contradicts our assumption on $f$. Therefore, $f$ is transverse to $H$.
\end{proof}

\begin{proof}[Proof of Theorem \ref{thm:relhypgps}]
 By Proposition \ref{lem:transverse}, there exists a hyperbolic element $f\in G$ which is transverse to $H$.  Thus Theorem \ref{thm:main1} implies that for any random subgroup $R$ of $G$, $\langle H,R\rangle\simeq H*R$ and $\langle H, R\rangle$ is quasi-convex in $\Cay(G,S\cup\mathcal P)$.  It remains to show that $\langle H, R\rangle$ is relatively quasi-convex.  
 
Let $R$ be generated by random walks $w_i,\dots,w_k$. We can assume that each $w_i$ is hyperbolic, and hence $\langle w_i\rangle$ is relatively quasi-convex (see, for example, \cite[Section 4.3]{Osin:RelHyp}). Fix $\sigma$ such that $H$ and each $\langle w_i\rangle $ is $\sigma$--relatively quasi-convex.
 Let  $x,y\in \langle H, R\rangle$, and let $p$ be a geodesic in $\Cay(G,S\cup\mathcal P)$ from $x$ to $y$.  Let $q$ be the path from $x$ to $y$ which is labeled by a word $W$ in $H$ and the generators of $R$. It is shown in the proof of Theorem \ref{t:words} that $q$ is a $(8, c)$--quasi-geodesic, where $c$ is independent of the given word $W$. By \cite[Proposition~3.15]{Osin:RelHyp}, there is a constant $\e'>0$ such that for any vertex $u$ in $p$, there is a vertex $v$ in $q$ such that $d_S(u,v)\leq \e'$.  
 
 If $v$ lies on a subgeodesic of $q$ labeled by some $h\in H$, then since $H$ is $\sigma$--relatively quasi-convex, we have $d_S(u,gH)\leq \e'+\sigma$ for some $g\in \langle H, R\rangle$, and hence $d_S(u, \langle H, R\rangle)\leq\e'+\sigma$. If $v$ lies on a subgeodesic of $q$ labeled by some $w_i$, then we can similarly obtain $d_S(u, \langle H, R\rangle)\leq\e'+\sigma$.  Therefore, $\langle H, R\rangle$ is relatively quasi-convex.
\end{proof}
\subsection{Mapping class groups} \label{sec:MCG}

Let $S$ be a surface of genus $g$ with $p$ punctures. $S$ is called \emph{exceptional} if $3g+p\leq 4$; otherwise $S$ is \emph{non-exceptional}. When $S$ is a non-exceptional surface, $\MCG(S)$ has a non-elementary, acylindrical action on a hyperbolic metric space called the curve graph of $S$, denoted $\mc{C}(S)$ \cite{Bow08, MasMin1}. Loxodromic elements for this action are the pseudo-Anosov elements of $MCG(S)$, and a subgroup $H\leq \MCG(S)$ is elliptic if and only if it contains no pseudo-Anosov elements by \cite[Theorem 1.1]{OsinAH}. Hence Corollary \ref{cor:mainell}  implies Corollary \ref{cor:elliptic}(3). Note that by Ivanov's Theorem \cite{Iva92}, these elliptic subgroups are precisely the subgroups which are either finite or which fix (set-wise) a finite collection of disjoint simple closed curves on $S$.  

We  now turn to the proof of Theorem \ref{thm:MCGCCC}, which will involve both Teichm{\"u}ller space $\mathcal T(S)$ of the surface $S$ with the Teichm{\"u}ller metric and the curve graph $\mathcal C(S)$.  For definitions and details about these spaces and the actions of the mapping class group, we refer the reader to \cite{Primer}.  We will use $d_{\mathcal T}$ for the Teichm{\"u}ller metric on $\mathcal T(S)$ and $d_{\mc{C}}$ for distance in $\mathcal C(S)$.

Fix a basepoint $x_0\in\mc T(S)$. A subgroup $H\leq \MCG(S)$ is \emph{convex cocompact} if there is a constant $\nu$ such that $Hx_0$ is $\nu$--quasi-convex in $\mc T(S)$ \cite{FarMos}.  Equivalently, $H$ is convex cocompact if the orbit map $H\to \mc C(S)$ is a quasi-isometric embedding by \cite{Ham05, KenLei}.  

We will need the following theorem from \cite{Rafi}.

\begin{thm}[{\cite[Theorem~8.1]{Rafi}}] \label{thm:thintriangles}
Let $x,y,z\in\mathcal T(S)$, and let $\mathcal G\colon [a,b]\to \mathcal T(S)$ be a Teichm{\"u}ller geodesic joining $x$ and $y$.  For all $\varepsilon\geq 0$, there exist constants $C,D$ such that the following holds.  Let $[c,d]$ be a subinterval of $[a,b]$ with $d-c>C$ such that for all $t\in[c,d]$, the geodesic $\mathcal G(t)$ is in the $\varepsilon$--thick part of $\mathcal T(S)$.  Then there exists $w\in [\mathcal G(c),\mathcal G(d)]$ such that $\min\{d_{\mathcal T}(w,[x,z]),d_{\mathcal T}(w,[y,z])\}\leq D$.
\end{thm}

We now show that the element $f$ provided by Corollary \ref{cor:MCG} is transverse to $H$ with respect to the action of $G$ on $\mc C(S)$. We first show that there is a uniform bound on the diameter of projections of cosets of $H$ to the axis of $f$ in $\mc{T}(S)$.  For a pseudo-Anosov $f\in G$ such that $H^g\cap\langle f\rangle=\{1\}$ for all $g\in G$. Let $\alpha_f$ be an axis of $f$ in $\mc T(S)$ and let $p_f\colon \mathcal T(S)\to\alpha_f$ be a nearest-point projection map, that is, $p_f(x)=\{y\in \alpha_f\mid d_\mc T(x,y)=d_\mc T(x,\alpha_f)\}$. 

\begin{lem}\label{lem:boundedTproj}
Let $G$ be a non-exceptional mapping class group of a surface $S$ and $H\leq G$ a convex cocompact subgroup.  There exists a pseudo-Anosov $f\in G$ and a constant $B=B(f,H)$ such that $\diam_{\mathcal T}(p_f(gH x_0))\leq B$ for all $g\in G$. 
\end{lem}

\begin{proof}

Let $\mathcal T(S)$ be the Teichm{\"u}ller space of $S$ with the Teichm{\"u}ller metric, and fix a basepoint $x_0\in \mathcal T(S)$.  The action of $G$ on $\mc T(S)$ is proper \cite{Primer}. Let $\nu$ be a constant so that $gHx_0$ is $\nu$--quasi-convex in $\mc T(S)$ for all $g\in G$.  By Corollary \ref{cor:MCG}, there exists a pseudo-Anosov $f\in G$ such that $H^g\cap\langle f\rangle=\{1\}$ for all $g\in G$. 

By \cite{Minsky}, the axis $\alpha_f$ is \emph{strongly contracting} in $\mathcal T(S)$, that is, there exist constants $B', K$ such that for all geodesics $\gamma$ at distance at least $K$ from $\alpha_f$, we have $\diam_{\mathcal T}(p_f(\gamma))\leq B'$. Thus it suffices to show that there exists a constant $B=B(f,H)$ such that for all $g\in G$ and all $x_1,x_2\in\alpha_f$ with $d_{\mathcal T}(x_i,gHx_0)\leq K$, we have $d_{\mathcal T}(x_1,x_2)\leq B$.  Suppose this is not the case.  Then for any $D$, there exists $g\in G$ and orbit points $f^{i_1}x_0,f^{i_2}x_0\in\alpha_f$ with $d_{\mathcal T}(f^{i_j}x_0, gHx_0)\leq K$ for $j=1,2$, and $d_{\mathcal T}(f^{i_1}x_0,f^{i_2}x_0)\geq D$.    Let $x_1,x_2$ be the nearest points in $gHx_0$ to $f^{i_1}x_0,f^{i_2}x_0$, respectively.  By choosing $D$ sufficiently large, we can ensure that the concatenation $\gamma=[f^{i_1}x_0,x_1]\cdot[x_1,x_2]\cdot[x_2,f^{i_2}x_0]$ is a uniform quasi-geodesic with constants depending only on the quasi-constants of $\alpha_f$.  Since $\alpha_f$ is strongly contracting, it is Morse \cite{ACGH}, and thus there is a constant $K'$ depending only on $K$ and the quasi-constants for $\alpha_f$ such that the Hausdorff distance between $\gamma$ and  $\alpha_f|_{[f^{i_1}x_0,f^{i_2}x_0]}$ is at most $K'$.  Moreover, since $gHx_0$ is $\nu$--quasi-convex, the geodesic from $x_1,x_2$ is contained in the $\nu$--neighborhood of $gHx_0$.  Therefore, for every $i_1\leq i\leq i_2$, we have $d_{\mathcal T}(f^ix_0,gHx_0)\leq K'+\nu$.  Let $gh_ix_0\in gHx_0$ be the nearest point in $gHx_0$ to $f^ix_0$, so that  $d_{\mathcal T}(x_0,f^{-i}gh_ix_0)\leq K+\nu$ for all $i$.  Since the action of $G$ on $\mathcal T(S)$ is proper and $D$ can be arbitrarily large, it follows that for some $i\neq j$, we have $f^{-i}gh_i=f^{-j}gh_j$.  Thus $f^{j-i}=gh_jh_i^{-1}g^{-1}$, which contradicts the fact that $H^g\cap\langle f\rangle=\{1\}$.
\end{proof}

There is a Lipschitz map $\phi\colon \mathcal T(S)\to\mathcal C(S)$ \cite{MasurMinsky}; we call the image of a geodesic under $\phi$ the \emph{shadow} of the geodesic.  By \cite[Lemma~3.3]{MasurMinskyII} (see also \cite[Theorem~B]{Rafi}), the shadow of any geodesic is an (unparametrized) $(a,b)$--quasi-geodesic, where $a$ and $b$ depend only on the surface $S$.  Let $\sigma$ be the Morse constant associated to $(a,b)$--quasi-geodesics in $\mc{C}(S)$.

Let $\beta_f$ be the shadow in $\mathcal C(S)$ of $\alpha_f$.  Then $\beta_f$ is an $(a,b)$--quasi-geodesic axis for $f$ in $\mathcal C(S)$. Since $\mc C(S)$ is a hyperbolic space and $\beta_f$ a quasi-geodesic, there is a coarsely well-defined nearest-point projection map from $\mathcal C(S)$ to $\beta_f$. We denote by $p_{\alpha_f}$ and $p_{\beta_f}$ the closest point projection maps to $\alpha_f$ and $\beta_f$ respectively. 

The following appears to be well-known to experts, but we include a proof for the sake of completeness.
\begin{lem}\label{lem:projcommute}
Let $G$ be a non-exceptional mapping class group of a surface $S$ and let $f$ be a pseudo-Anosov element of $G$. Then there is a constant $A$ such that for any $x\in \mathcal{T}(S)$, $d_{\mathcal{C}}(\phi(p_{\alpha_f}(x)), p_{\beta_f}(\phi(x)))\leq A$.

\end{lem}

\begin{proof}
Fix $\varepsilon>0$, and let $C,D$ be the constants provided by Theorem \ref{thm:thintriangles}.   

Let $\beta_f$ be the shadow in $\mathcal C(S)$ of $\alpha_f$. We identify each point on $\alpha_f$ with its image under $\phi$ on $\beta_f$. Fix a point $y_0\in\mc C(S)$ and let $\pi\colon G\to\mc C(S)$ be the corresponding orbit map. Let $z_1=\phi(p_{\alpha_f}(x))$ and $z_2=p_{\beta_f}(\phi(x))$.

 If $d_\mc{C}(z_1, z_2)$ is sufficiently large, then there is a subpath $I$ of $\beta_f$ between $z_1$ and $z_2$ of length at least $C$ such that $d_{\mc C}(I,z_i)> 2(D+\sigma)$ for $i=1,2$.  Since $\phi\colon \mathcal T(S)\to\mathcal C(S)$ is Lipschitz, we have $d_{\mathcal T}(I,z_i)> 2(D+\sigma)$ for $i=1,2$ and the length of $I$ in $\mathcal T(S)$ is at least $C$.  Since $\alpha_f$ is contained in the $\varepsilon$--thick part of $\mathcal T(S)$, it follows from Theorem \ref{thm:thintriangles} that there is a point $w\in I$ and a point $u\in[x,z_1]\cup[x,z_2]$ with $d_{\mathcal T}(w,u)\leq D$, where here $[x,z_i]$ are geodesics in $\mathcal T(S)$.  Note that $d_\mc T(w,z_i)>2(D+\sigma)$ for $i=1,2$.   If $u\in[x,z_1]$, then by the triangle inequality, we have \[d_{\mathcal T}(u,z_1)> 2(D+\sigma)-D>D,\] which contradicts the fact that $p_{\alpha_f}(x)=z_1$.  Thus $u\in[x,z_2]$.  See Figure \ref{fig:Teichproj}.  
 
 \begin{figure}[h!]
\centering{
\begin{subfigure}{.4\textwidth}
\def\svgwidth{2.5in}  
  \centering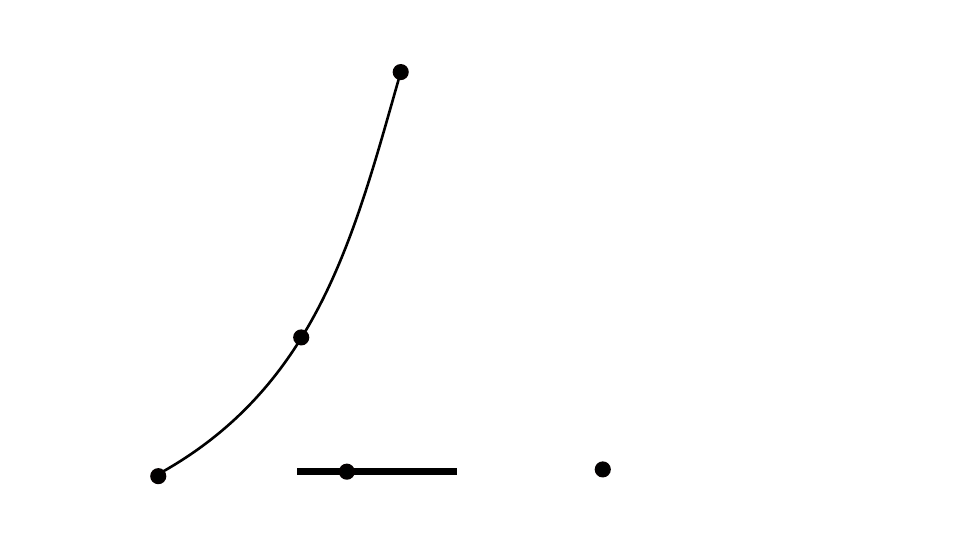 \\
	\caption{In $\mc T(S)$: $u$ is on $[x,z_2]$; the bold interval on $\alpha_f$ is $I$.} 
	\label{fig:Teichproj}
\end{subfigure} \hspace{.5in}
\begin{subfigure}{.4\textwidth}
\def\svgwidth{3in}  
  \centering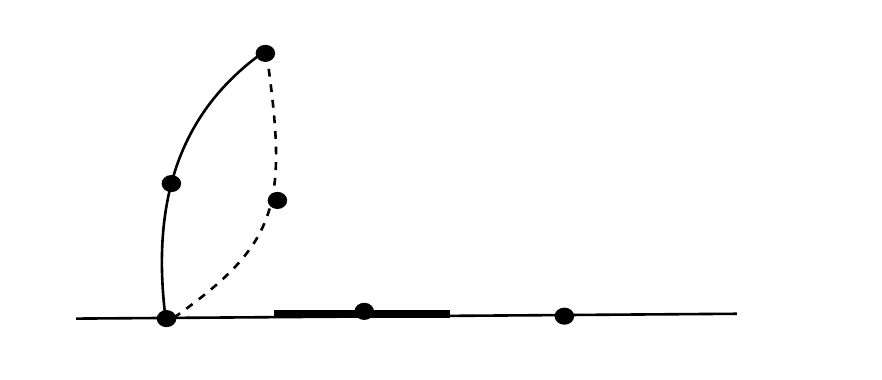 \\
	\caption{In $\mc C(S)$: the dotted path is the  shadow of $[x,z_2]$, while the bold interval on $\beta_f$ is $I$.} 
	\label{fig:CSproj}
\end{subfigure}}
\caption{The proof of Lemma \ref{lem:projcommute}.}
\end{figure}

The shadow of $[x, z_2]$ is an (unparametrized) $(a,b)$--quasi-geodesic, so there is a point $u'$ on a geodesic in $\mathcal C(S)$ from $\phi(x)$ to $z_2$ such that $d_{\mc{C}}(u,u')\leq\sigma$.   Combining this with the triangle inequality and the fact that $\phi$ is Lipschitz gives that $d_{\mc{C}}(u',\beta_f)\leq d_{\mc{C}}(u',w)\leq D+\sigma$.  However, since $d_{\mc{C}}(z_2,w)> 2(D+\sigma)$, it follows from the triangle inequality that \[d_{\mc{C}}(u',z_2)> 2(D+\sigma)-(D+\sigma)>D+\sigma,\] which contradicts the fact that $z_2$ is a nearest-point projection of $\phi(x)$ onto $\beta_f$.  See Figure \ref{fig:CSproj}.
\end{proof}

\begin{prop} \label{lem:MCGftrans}
Let $G$ be a non-exceptional mapping class group of a surface $S$ and $H\leq G$ a convex cocompact subgroup.  There exists a pseudo-Anosov $f\in G$ such that $f$ is transverse to $H$ with respect to the action of $G$ on the curve graph of $S$.
\end{prop}

\begin{proof}

Let $f$ be the pseudo-Anosov element given by Lemma \ref{lem:boundedTproj}, and let $\alpha_f$ and $\beta_f$ be axis for $f$ in $\mc{T}(S)$ and $\mc{C}(S)$ respectively. Let $K\geq 0$, and recall that we want to find a bound on $\diam(\beta_f\cap\mc{N}_K(g\pi(H)))$ which holds for all $g\in G$.

Note that the projection of $\pi(gH)$ onto $\beta_f$ has size at least $\diam(\beta_f\cap\mc{N}_K(g\pi(H)))-4K$. Moreover, the diameter of this projection $p_{\beta_f}(\pi(gH))$ is at most $B+2A$ by Lemmas \ref{lem:projcommute} and \ref{lem:boundedTproj}.  This provides a uniform bound on $\diam(\beta_f\cap\mc{N}_K(g\pi(H)))$, completing the proof.
\end{proof}

\begin{proof}[Proof of Theorem \ref{thm:MCGCCC}]
In light of Proposition \ref{lem:MCGftrans}, this now follows directly from Theorem \ref{thm:main1}.
\end{proof}

\subsection{$\operatorname{Out}(\mathbb F_n)$}   
We consider the action of $\operatorname{Out}(\mathbb F_n)$ on the \emph{free factor graph} $\mc{FF}_n$, the graph whose vertices are conjugacy classes of free factors of $\mathbb F_n$ with edges given by inclusion.  This is an infinite diameter hyperbolic graph with a WPD action of $\operatorname{Out}(\mathbb F_n)$, and an element of $\operatorname{Out}(\mathbb F_n)$ acts loxodromically on $\mc{FF}_n$ if and only if it is fully irreducible  \cite{BesFei14}.  Any subgroup $H$ which virtually fixes a free factor of $\mathbb F_n$ up to conjugacy clearly acts elliptically on $\mc{FF}_n$, and when $H$ is finitely generated the converse is also true \cite{HanMos}. Hence Corollary \ref{cor:mainell} implies Corollary \ref{cor:elliptic}(4).   We note that when $H$ is finitely generated, it acts elliptically on $\mc{FF}_n$ if and only if it does not contain a fully irreducible element by \cite[Theorem~A]{HanMos}.

Following \cite{HamHen18} we call a subgroup $H\leq \Out(\mathbb F_n)$ \emph{convex cocompact} if $H$ is quasi-isometrically embedded in $\mc{FF}_n$ under the orbit map. An $\Out(\mathbb F_n)$--analogue of Theorem \ref{thm:MCGCCC} holds for these subgroups with essentially the same proof with the curve complex replaced by $\mc{FF}_n$ and Teichm\"uller space replaced by Culler--Vogtmann's outer space $\operatorname{CV}_n$ (see \cite{vogtmann:CVnsurvey} and references therein for the definition and details of $\operatorname{CV}_n$).  Indeed, if $H$ is quasi-isometrically embedded in $\mc{FF}_n$ under the orbit map, then $H$ is stable in $\Out(\mathbb F_n)$ by \cite{AouDurTay}. Hence by Theorem \ref{thm:stablesubgroup} there exists a fully irreducible element $f$ such that no power of $f$ is conjugate into $H$. Moreover, an orbit $Hx_0$ in $\operatorname{CV}_n$ satisfies the following property: for any $g,h\in H$, the points $gx_0,hx_0$ in $\operatorname{CV}_n$ are connected by a uniformly strongly Morse $c$--coarse geodesic which is contained in a uniformly bounded neighborhood of $Hx_0$ \cite[Theorem~3]{HamHen18}.  Using this and the fact that the axis of $f$ is strongly contracting in $\operatorname{CV}_n$ by \cite{Alg}, one can repeat the proof of Lemma \ref{lem:boundedTproj} to find a uniform bound on the diameter of projections of cosets of $H$ to the axis of $f$ in outer space. There is again a Lipschitz map from outer space to $\mc{FF}_n$, and the analogue of Lemma \ref{lem:projcommute} is given by \cite[Lemma 4.2]{DowTay}. Using these lemmas, $f$ can be shown to be transverse to $H$ in the same way as in Proposition \ref{lem:MCGftrans}. Thus we can apply Theorem \ref{thm:main1} to obtain the following.

\begin{thm}\label{thm:outfn}
Let $H$ be a convex cocompact subgroup of $\Out(\mathbb F_n)$ where $n\geq 3$.  Let $(\mu_i)$ be a sequence of permissible probability distributions on $\Out(\mathbb F_n)$. Then a random subgroup $R$ of $\Out(\mathbb F_n)$satisfies $\langle H, R\rangle\cong H\ast R$ and $\langle H, R\rangle$ is convex cocompact.
\end{thm}

\subsection{Other applications}

In this section, we finish the proof of Corollary \ref{cor:elliptic} by applying Theorem \ref{thm:main1} in various contexts. Note that parts (1)--(4) of Corollary \ref{cor:elliptic} are proved above, so it remains to prove (5)--(7).  In addition, we conjecture some further applications of Theorem \ref{thm:main1} and finally prove that the subgroup produced by applying Corollary \ref{cor:mainell} to a small subgroup of an acylindrically hyperbolic group is again small (Corollary \ref{cor:small}).

\subsubsection*{Artin groups}
We refer to \cite{KimKob14,CalvezWiest1} for definitions and notation that we use for Artin groups.
First, let $G$ be a directly indecomposable right-angled Artin group $A(\Gamma)$. We will assume that $\Gamma$ does not split as a join of two subgraphs, or equivalently that $A(\Gamma)$ does not decompose as a non-trivial direct product. In this case we consider the action of $G$ on the \emph{extension graph} $\Gamma^e$ of $G$.  Defined by Kim--Koberda, this is an infinite diameter hyperbolic graph  whose vertices are conjugates of generators of $G$ with and edge between two vertices is the corresponding conjugates of generators commute \cite{KimKob13}.  Moreover, $G$ acts acylindrically on $\Gamma^e$ \cite{KimKob14}.  

Suppose $H\leq G$ is conjugate to a subgroup of $G$ whose support is contained in a subjoin of $\Gamma$.  Then the star-length $\|h\|_*$ of any $h\in H$ is uniformly bounded by \cite[Lemma~34]{KimKob14}.  For any vertex $v\in V(\Gamma^e)$, $d_{\Gamma^e}(v,v^h)$ is coarsely equal to $\|h\|_*$, and therefore the orbit $Hv$ has bounded diameter.  It follows that $H$ is elliptic in this action. Therefore, we may apply Corollary \ref{cor:mainell}, which proves Corollary \ref{cor:elliptic} (5).

By \cite{KobManTay}, when $\Gamma$ is connected and not a join, a subgroup $H$ of $A(\Gamma)$ is stable if and only if it quasi-isometrically embeds into $\Gamma^e$ under the orbit map.  

\begin{conj}\label{conj:raags}
Let $H$ be a stable subgroup of a right-angled Artin group $A(\Gamma)$, where $\Gamma$ is connected and not a join.  Let $(\mu_i)$ be a sequence of permissible probability distributions on $A(\Gamma)$. Then a random subgroup $R$ of $A(\Gamma)$ satisfies $\langle H, R\rangle\cong H\ast R$ and $\langle H, R\rangle$ is stable in $A(\Gamma)$.
\end{conj}

In order to resolve Conjecture \ref{conj:raags}, it suffices to show that there is an element $f\in G$ which is transverse to $H$ in $\Gamma^e$.  For relatively hyperbolic groups and mapping class groups, the element $f$ whose existence was guaranteed by Theorem \ref{thm:stablesubgroup} could be shown to be transverse to the appropriate subgroup.  We expect this to also be the case when $G$ is a right-angled Artin group.

Next suppose $A=A(\Gamma)$ is an irreducible Artin group of spherical type such that $A\neq A_1,A_2,I_{2m}$, and let $Z(A)$ denote the center of $A$.  We will consider the \emph{additional length graph}, denoted $\mc C_{AL}(A)$, defined by Calvez--Wiest \cite{CalvezWiest1}.  They show that $\mc C_{AL}(A)$ is an infinite diameter hyperbolic graph with a partially WPD action of $A/Z(A)$ in \cite[Theorem~A]{CalvezWiest1} and \cite[Theorems~1.1\&~1.3]{CalvezWiest2}.  A subgroup of $A$ is a \emph{parabolic subgroup} if it is conjugate to $A(\Gamma')$ for some $\Gamma'\subseteq \Gamma$.  Antolin--Cumplido show in \cite[Theorem~2]{AntCum} that if $H\leq A$ normalizes a parabolic subgroup, then $H$ acts elliptically on $\mc C_{AL}(A)$. Therefore we may apply Corollary \ref{cor:mainell} to prove Corollary \ref{cor:elliptic} (6).

\subsubsection*{$3$--manifold groups}
We refer to \cite{ASW} for the definitions and notation that we use for $3$--manifolds and their fundamental groups. Let $M$ be a closed, orientable, irreducible, non-geometric $3$--manifold and $N$ a JSJ-component of $M$.  Let $G=\pi_1(M)$ and $H$ be a subgroup of $\pi_1(N)$.  We consider the action of $G$ on the Bass--Serre tree of $\pi_1(M)$ associated to the splitting of $\pi_1(M)$ defined by the JSJ-decomposition of $M$.  This is an infinite diameter tree with an acylindrical action of $\pi_1(M)$ by \cite[Lemma~2.4]{WilZal} and \cite[Lemma~5.2]{MinOsi}.  Moreover, $\pi_1(N)$ is the stabilizer of a vertex under this action, hence $H$ also fixes a vertex. Therefore we may apply Corollary \ref{cor:mainell} to get Corollary \ref{cor:elliptic}(7). This completes the proof of Corollary \ref{cor:elliptic}.

\subsubsection*{Hierarchically hyperbolic groups}

We refer to \cite{BHS1, BHS2,SisHHS} for definitions and background on hierarchically hyperbolic groups. What is relevant for us is that such groups share many properties with mapping class groups; in particular, a hierarchically hyperbolic group $G$ admits an acylindrical action on a hyperbolic metric space $X$ such that a subgroup $H\leq G$ is stable in $G$ if and only if it is quasi-isometrically embedded in $X$ under the orbit map \cite{ABD}. We conjecture that the natural analogue of Theorem \ref{thm:MCGCCC} also holds in this setting.

\begin{conj}\label{conj:hhg}
Let $G$ be a hierarchically hyperbolic group, and let $H$ be a stable subgroup of $G$.  Let $(\mu_i)$ be a sequence of permissible probability distributions on $G$. Then a random subgroup $R$ of $G$ satisfies $\langle H, R\rangle\cong H\ast R$ and $\langle H, R\rangle$ is stable in $G$.
\end{conj}

Note that since right-angled Artin groups are hierarchically hyperbolic, Conjecture \ref{conj:hhg} implies Conjecture \ref{conj:raags}.

\subsubsection*{Small subgroups}

The following notion of a small subgroup of an acylindrically hyperbolic group was introduced in \cite{HO16}. 
\begin{defn}\cite[Definition 2.10]{HO16}\label{defn:small}
Let $G$ be acylindrically hyperbolic. A subgroup $H\leq G$ is called \emph{small} if $G$ has a generating set $S$ such that $H\subset S$, $\Cay(G, S)$ is hyperbolic, and the action of $G$ on $\Cay(G, S)$ is non-elementary and acylindrical.
\end{defn}

It is easy to see that $H$ is small if and only if $G$ admits a cobounded, acylindrical action on a hyperbolic metric space $X$ for which $H$ is elliptic. Note that when $G$ is acylindrically hyperbolic it will have many different cobounded, acylindrical actions on hyperbolic metric spaces \cite{ABO}. Examples of small subgroups include all subgroup of $G$ which are not acylindrically hyperbolic, for example all virtually cyclic subgroups and all finite subgroups, as well as any subgroup which is hyperbolically embedded  in $G$.


\begin{cor} \label{cor:small}
Let $G$ be an acylindrically hyperbolic group and $H$ a small subgroup of $G$. Let $(\mu_i)$ be a sequence of finitely supported probability distributions of full support. Then a random subgroup $R$ will satisfy $\langle H, R\rangle\cong H\ast R$ and $\langle H, R\rangle$ is small in $G$.
\end{cor}

The first statement follows immediately from Corollary \ref{cor:mainell}; thus it remains to show that $\langle H, R\rangle$ is small in $G$. We will use a more complicated criteria for a subgroup to be small from \cite{HO16}.   

\begin{prop}[{\cite[Proposition 2.12]{HO16}}]\label{p:small}
Let $T\subset G$ and $F\leq G$ such that $F$ hyperbolically embeds in $G$ with respect to $T$.  Suppose $H$ is a subgroup of $G$ generated by a set $Z$ such that $\sup_{z\in Z} |z|_T<\infty$ and $d_Z(h_1, h_2)\leq K d_{T\cup E(g)}(h_1, h_2)$ for all $h_1, h_2\in H$. Then $H$ is small in $G$.
\end{prop}

\begin{proof}[Proof of Corollary \ref{cor:small}] Let $H$ be a small subgroup, $S$  a generating set for $G$ as in Definition \ref{defn:small}, and $g$  an element of $G$ which acts loxodromically on $\Cay(G, S)$. By \cite[Corollary 3.17]{H16}, $E(g)$ hyperbolically embeds in $G$ with respect to $S$. By \cite[Theorem 5.4]{OsinAH},  there is a generating set $T$ such that $S\subseteq T$, the graph $\Cay(G, T\cup E(g))$ is hyperbolic, and the action of $G$ on $\Cay(G, T\cup E(g))$ is non-elementary and acylindrical.

Note that if the image of $H$ in $G/E(G)$ is small, then $H$ will be small in $G$. Hence we can assume that $E(G)=\{1\}$.

 Let $w_1,\dots,w_k$ be random walks in $G$ with respect to $\mu_1,\dots,\mu_k$ respectively and let $R=\langle w_1,\dots,w_k\rangle$. 
Let $Z=H\cup\{w_1,\dots,w_k\}$. Since $H\subseteq S\subseteq T$, clearly $\sup_{z\in Z} |z|_{T}<\infty$. By Theorem \ref{t:words} with $X=\Cay(G, T\cup E(g))$, we have that for all $h\in \langle H, R\rangle$
\[
|h|_Z\leq8|h|_{T\cup E(g)}+c\leq (8+c)|h|_{T\cup E(g)}.
\]
Therefore, Proposition \ref{p:small} implies that $\langle H, R\rangle$ is small in $G$.
\end{proof}


\vspace{1cm}

\end{document}

%% file: qcxity2.pdf_tex
\begingroup%
  \makeatletter%
  \providecommand\color[2][]{%
    \errmessage{(Inkscape) Color is used for the text in Inkscape, but the package 'color.sty' is not loaded}%
    \renewcommand\color[2][]{}%
  }%
  \providecommand\transparent[1]{%
    \errmessage{(Inkscape) Transparency is used (non-zero) for the text in Inkscape, but the package 'transparent.sty' is not loaded}%
    \renewcommand\transparent[1]{}%
  }%
  \providecommand\rotatebox[2]{#2}%
  \ifx\svgwidth\undefined%
    \setlength{\unitlength}{248.13015384bp}%
    \ifx\svgscale\undefined%
      \relax%
    \else%
      \setlength{\unitlength}{\unitlength * \real{\svgscale}}%
    \fi%
  \else%
    \setlength{\unitlength}{\svgwidth}%
  \fi%
  \global\let\svgwidth\undefined%
  \global\let\svgscale\undefined%
  \makeatother%
  \begin{picture}(1,0.79253379)%
    \put(0,0){\includegraphics[width=\unitlength,page=1]{qcxity2.pdf}}%
    \put(0.84639022,0.75579022){\color[rgb]{0,0,0}\makebox(0,0)[lb]{\smash{$q$}}}%
    \put(0.79198324,0.34471568){\color[rgb]{0,0,0}\makebox(0,0)[lb]{\smash{$\alpha$}}}%
    \put(0.23280102,0.03036453){\color[rgb]{0,0,0}\makebox(0,0)[lb]{\smash{$\alpha_f$}}}%
  \end{picture}%
\endgroup%

%% file: CaseB.pdf_tex
\begingroup%
  \makeatletter%
  \providecommand\color[2][]{%
    \errmessage{(Inkscape) Color is used for the text in Inkscape, but the package 'color.sty' is not loaded}%
    \renewcommand\color[2][]{}%
  }%
  \providecommand\transparent[1]{%
    \errmessage{(Inkscape) Transparency is used (non-zero) for the text in Inkscape, but the package 'transparent.sty' is not loaded}%
    \renewcommand\transparent[1]{}%
  }%
  \providecommand\rotatebox[2]{#2}%
  \ifx\svgwidth\undefined%
    \setlength{\unitlength}{263.68944742bp}%
    \ifx\svgscale\undefined%
      \relax%
    \else%
      \setlength{\unitlength}{\unitlength * \real{\svgscale}}%
    \fi%
  \else%
    \setlength{\unitlength}{\svgwidth}%
  \fi%
  \global\let\svgwidth\undefined%
  \global\let\svgscale\undefined%
  \makeatother%
  \begin{picture}(1,0.60847052)%
    \put(0,0){\includegraphics[width=\unitlength,page=1]{CaseB.pdf}}%
    \put(-0.00377753,0.0334866){\color[rgb]{0,0,0}\makebox(0,0)[lb]{\smash{$x_{l-1}$}}}%
    \put(0.75848278,0.01073262){\color[rgb]{0,0,0}\makebox(0,0)[lb]{\smash{$x_{l+1}$}}}%
    \put(0.4029509,0.57389505){\color[rgb]{0,0,0}\makebox(0,0)[lb]{\smash{$x_l$}}}%
    \put(0.13559097,0.30653508){\color[rgb]{0,0,0}\makebox(0,0)[lb]{\smash{$q_l$}}}%
    \put(0.51672117,0.48856739){\color[rgb]{0,0,0}\makebox(0,0)[lb]{\smash{$\beta_1$}}}%
    \put(0.57645046,0.29515805){\color[rgb]{0,0,0}\makebox(0,0)[lb]{\smash{$\beta_2$}}}%
    \put(0.47974578,0.3804857){\color[rgb]{0,0,0}\makebox(0,0)[lb]{\smash{$\eta$}}}%
    \put(0,0){\includegraphics[width=\unitlength,page=2]{CaseB.pdf}}%
  \end{picture}%
\endgroup%

%% file: quad.pdf_tex
\begingroup%
  \makeatletter%
  \providecommand\color[2][]{%
    \errmessage{(Inkscape) Color is used for the text in Inkscape, but the package 'color.sty' is not loaded}%
    \renewcommand\color[2][]{}%
  }%
  \providecommand\transparent[1]{%
    \errmessage{(Inkscape) Transparency is used (non-zero) for the text in Inkscape, but the package 'transparent.sty' is not loaded}%
    \renewcommand\transparent[1]{}%
  }%
  \providecommand\rotatebox[2]{#2}%
  \ifx\svgwidth\undefined%
    \setlength{\unitlength}{546.41507024bp}%
    \ifx\svgscale\undefined%
      \relax%
    \else%
      \setlength{\unitlength}{\unitlength * \real{\svgscale}}%
    \fi%
  \else%
    \setlength{\unitlength}{\svgwidth}%
  \fi%
  \global\let\svgwidth\undefined%
  \global\let\svgscale\undefined%
  \makeatother%
  \begin{picture}(1,0.32129015)%
    \put(0,0){\includegraphics[width=\unitlength,page=1]{quad.pdf}}%
    \put(0.16125617,0.19996221){\color[rgb]{0,0,0}\makebox(0,0)[lb]{\smash{$q_1$}}}%
    \put(0.55944116,0.19860322){\color[rgb]{0,0,0}\makebox(0,0)[lb]{\smash{$q_2$}}}%
    \put(0.12456339,0.0803708){\color[rgb]{0,0,0}\makebox(0,0)[lb]{\smash{$1$}}}%
    \put(0.61923687,0.0776529){\color[rgb]{0,0,0}\makebox(0,0)[lb]{\smash{$h_n^k$}}}%
    \put(-0.00182296,0.3032457){\color[rgb]{0,0,0}\makebox(0,0)[lb]{\smash{$g_n$}}}%
    \put(0.73611023,0.30460469){\color[rgb]{0,0,0}\makebox(0,0)[lb]{\smash{$g_n(y^nh^n)^{kl_n}$}}}%
    \put(0,0){\includegraphics[width=\unitlength,page=2]{quad.pdf}}%
    \put(0.28492455,0.25703994){\color[rgb]{0,0,0}\makebox(0,0)[lb]{\smash{$W^n$}}}%
    \put(0.36646412,0.25568093){\color[rgb]{0,0,0}\makebox(0,0)[lb]{\smash{$U^n$}}}%
    \put(0.27677062,0.17414135){\color[rgb]{0,0,0}\makebox(0,0)[lb]{\smash{$u_n$}}}%
    \put(0.35423322,0.17414135){\color[rgb]{0,0,0}\makebox(0,0)[lb]{\smash{$v_n$}}}%
    \put(0.43441372,0.1727824){\color[rgb]{0,0,0}\makebox(0,0)[lb]{\smash{$z_n$}}}%
    \put(0,0){\includegraphics[width=\unitlength,page=3]{quad.pdf}}%
    \put(0.04981881,0.03416506){\color[rgb]{0,0,0}\makebox(0,0)[lb]{\smash{$H$}}}%
    \put(0,0){\includegraphics[width=\unitlength,page=4]{quad.pdf}}%
  \end{picture}%
\endgroup%

%% file: relhyp2.pdf_tex
\begingroup%
  \makeatletter%
  \providecommand\color[2][]{%
    \errmessage{(Inkscape) Color is used for the text in Inkscape, but the package 'color.sty' is not loaded}%
    \renewcommand\color[2][]{}%
  }%
  \providecommand\transparent[1]{%
    \errmessage{(Inkscape) Transparency is used (non-zero) for the text in Inkscape, but the package 'transparent.sty' is not loaded}%
    \renewcommand\transparent[1]{}%
  }%
  \providecommand\rotatebox[2]{#2}%
  \ifx\svgwidth\undefined%
    \setlength{\unitlength}{460.76198866bp}%
    \ifx\svgscale\undefined%
      \relax%
    \else%
      \setlength{\unitlength}{\unitlength * \real{\svgscale}}%
    \fi%
  \else%
    \setlength{\unitlength}{\svgwidth}%
  \fi%
  \global\let\svgwidth\undefined%
  \global\let\svgscale\undefined%
  \makeatother%
  \begin{picture}(1,0.47563112)%
    \put(0,0){\includegraphics[width=\unitlength,page=1]{relhyp2.pdf}}%
    \put(0.85862838,0.42003367){\color[rgb]{0,0,0}\makebox(0,0)[lb]{\smash{$\alpha_f'$}}}%
    \put(0,0){\includegraphics[width=\unitlength,page=2]{relhyp2.pdf}}%
    \put(0.02685401,0.02123776){\color[rgb]{0,0,0}\makebox(0,0)[lb]{\smash{$gH$}}}%
    \put(0,0){\includegraphics[width=\unitlength,page=3]{relhyp2.pdf}}%
    \put(0.17660598,0.22633279){\color[rgb]{0,0,0}\makebox(0,0)[lb]{\smash{$p$}}}%
    \put(0,0){\includegraphics[width=\unitlength,page=4]{relhyp2.pdf}}%
    \put(0.67143844,0.30120878){\color[rgb]{0,0,0}\makebox(0,0)[lb]{\smash{$\leq K$}}}%
    \put(0,0){\includegraphics[width=\unitlength,page=5]{relhyp2.pdf}}%
    \put(0.04801462,0.30446426){\color[rgb]{0,0,0}\makebox(0,0)[lb]{\smash{$\leq K$}}}%
    \put(0,0){\includegraphics[width=\unitlength,page=6]{relhyp2.pdf}}%
    \put(0.07080296,0.13680718){\color[rgb]{0,0,0}\makebox(0,0)[lb]{\smash{$x'$}}}%
    \put(0.68120488,0.1286685){\color[rgb]{0,0,0}\makebox(0,0)[lb]{\smash{$y'$}}}%
    \put(0.30682501,0.27842044){\color[rgb]{0,0,0}\makebox(0,0)[lb]{\smash{$a'$}}}%
    \put(0.50866454,0.27842042){\color[rgb]{0,0,0}\makebox(0,0)[lb]{\smash{$b'$}}}%
    \put(0.29543086,0.38747891){\color[rgb]{0,0,0}\makebox(0,0)[lb]{\smash{$a$}}}%
    \put(0.52982514,0.37771247){\color[rgb]{0,0,0}\makebox(0,0)[lb]{\smash{$b$}}}%
    \put(0.08545261,0.44770524){\color[rgb]{0,0,0}\makebox(0,0)[lb]{\smash{$x$}}}%
    \put(0.66004424,0.39236211){\color[rgb]{0,0,0}\makebox(0,0)[lb]{\smash{$y$}}}%
    \put(0,0){\includegraphics[width=\unitlength,page=7]{relhyp2.pdf}}%
    \put(0.24334322,0.00821588){\color[rgb]{0,0,0}\makebox(0,0)[lb]{\smash{$c$}}}%
    \put(0.63562816,0.09448599){\color[rgb]{0,0,0}\makebox(0,0)[lb]{\smash{$d$}}}%
    \put(0,0){\includegraphics[width=\unitlength,page=8]{relhyp2.pdf}}%
    \put(0.31984689,0.21005543){\color[rgb]{0,0,0}\makebox(0,0)[lb]{\smash{$\leq\sigma$}}}%
    \put(0,0){\includegraphics[width=\unitlength,page=9]{relhyp2.pdf}}%
    \put(0.4354164,0.12052983){\color[rgb]{0,0,0}\makebox(0,0)[lb]{\smash{$\beta$}}}%
    \put(0,0){\includegraphics[width=\unitlength,page=10]{relhyp2.pdf}}%
    \put(0.41751124,0.45584393){\color[rgb]{0,0,0}\makebox(0,0)[lb]{\smash{$\leq 2\delta+M$}}}%
    \put(0,0){\includegraphics[width=\unitlength,page=11]{relhyp2.pdf}}%
    \put(0.4077448,0.36631831){\color[rgb]{0,0,0}\makebox(0,0)[lb]{\smash{$q$}}}%
    \put(0.54284701,0.21982187){\color[rgb]{0,0,0}\makebox(0,0)[lb]{\smash{$\leq \sigma$}}}%
    \put(0.43704412,0.28167592){\color[rgb]{0,0,0}\makebox(0,0)[lb]{\smash{}}}%
    \put(0.41100028,0.29469782){\color[rgb]{0,0,0}\makebox(0,0)[lb]{\smash{$p'$}}}%
  \end{picture}%
\endgroup%

%% file: Teichproj2.pdf_tex
\begingroup%
  \makeatletter%
  \providecommand\color[2][]{%
    \errmessage{(Inkscape) Color is used for the text in Inkscape, but the package 'color.sty' is not loaded}%
    \renewcommand\color[2][]{}%
  }%
  \providecommand\transparent[1]{%
    \errmessage{(Inkscape) Transparency is used (non-zero) for the text in Inkscape, but the package 'transparent.sty' is not loaded}%
    \renewcommand\transparent[1]{}%
  }%
  \providecommand\rotatebox[2]{#2}%
  \ifx\svgwidth\undefined%
    \setlength{\unitlength}{276.08789887bp}%
    \ifx\svgscale\undefined%
      \relax%
    \else%
      \setlength{\unitlength}{\unitlength * \real{\svgscale}}%
    \fi%
  \else%
    \setlength{\unitlength}{\svgwidth}%
  \fi%
  \global\let\svgwidth\undefined%
  \global\let\svgscale\undefined%
  \makeatother%
  \begin{picture}(1,0.55847663)%
    \put(0,0){\includegraphics[width=\unitlength,page=1]{Teichproj2.pdf}}%
    \put(0.39928102,0.52545386){\color[rgb]{0,0,0}\makebox(0,0)[lb]{\smash{$x$}}}%
    \put(0.13465394,0.01025064){\color[rgb]{0,0,0}\makebox(0,0)[lb]{\smash{$z_2$}}}%
    \put(0.5910303,0.01596476){\color[rgb]{0,0,0}\makebox(0,0)[lb]{\smash{$z_1$}}}%
    \put(0.3349275,0.21651922){\color[rgb]{0,0,0}\makebox(0,0)[lb]{\smash{$u$}}}%
    \put(0.11591927,0.29361241){\color[rgb]{0,0,0}\makebox(0,0)[lb]{\smash{}}}%
    \put(0.32237528,0.00744054){\color[rgb]{0,0,0}\makebox(0,0)[lb]{\smash{$w$}}}%
    \put(0,0){\includegraphics[width=\unitlength,page=2]{Teichproj2.pdf}}%
    \put(0.37670582,0.15413286){\color[rgb]{0,0,0}\makebox(0,0)[lb]{\smash{$\leq D$}}}%
    \put(0.77603515,0.0183065){\color[rgb]{0,0,0}\makebox(0,0)[lb]{\smash{$\alpha_f$}}}%
    \put(0,0){\includegraphics[width=\unitlength,page=3]{Teichproj2.pdf}}%
    \put(-0.00360789,0.47739943){\color[rgb]{0,0,0}\makebox(0,0)[lb]{\smash{$\mc T(S):$}}}%
  \end{picture}%
\endgroup%

%% file: CSproj.pdf_tex
\begingroup%
  \makeatletter%
  \providecommand\color[2][]{%
    \errmessage{(Inkscape) Color is used for the text in Inkscape, but the package 'color.sty' is not loaded}%
    \renewcommand\color[2][]{}%
  }%
  \providecommand\transparent[1]{%
    \errmessage{(Inkscape) Transparency is used (non-zero) for the text in Inkscape, but the package 'transparent.sty' is not loaded}%
    \renewcommand\transparent[1]{}%
  }%
  \providecommand\rotatebox[2]{#2}%
  \ifx\svgwidth\undefined%
    \setlength{\unitlength}{253.93252294bp}%
    \ifx\svgscale\undefined%
      \relax%
    \else%
      \setlength{\unitlength}{\unitlength * \real{\svgscale}}%
    \fi%
  \else%
    \setlength{\unitlength}{\svgwidth}%
  \fi%
  \global\let\svgwidth\undefined%
  \global\let\svgscale\undefined%
  \makeatother%
  \begin{picture}(1,0.42991531)%
    \put(0,0){\includegraphics[width=\unitlength,page=1]{CSproj.pdf}}%
    \put(0.29416616,0.39401133){\color[rgb]{0,0,0}\makebox(0,0)[lb]{\smash{$\phi(x)$}}}%
    \put(0.60866349,0.01934929){\color[rgb]{0,0,0}\makebox(0,0)[lb]{\smash{$z_1$}}}%
    \put(0.78095345,0.03028829){\color[rgb]{0,0,0}\makebox(0,0)[lb]{\smash{$\beta_f$}}}%
    \put(0.36800466,0.02481879){\color[rgb]{0,0,0}\makebox(0,0)[lb]{\smash{$w$}}}%
    \put(0.14648914,0.011145){\color[rgb]{0,0,0}\makebox(0,0)[lb]{\smash{$z_2$}}}%
    \put(0.33518754,0.20531287){\color[rgb]{0,0,0}\makebox(0,0)[lb]{\smash{$u$}}}%
    \put(0.13555008,0.21078246){\color[rgb]{0,0,0}\makebox(0,0)[lb]{\smash{$u'$}}}%
    \put(0,0){\includegraphics[width=\unitlength,page=2]{CSproj.pdf}}%
    \put(0.38714799,0.15608721){\color[rgb]{0,0,0}\makebox(0,0)[lb]{\smash{$\leq D$}}}%
    \put(-0.00392267,0.34478566){\color[rgb]{0,0,0}\makebox(0,0)[lb]{\smash{$C(S):$}}}%
  \end{picture}%
\endgroup%

%% file: Random_walks_in_AH_GAFA.bbl
\begin{thebibliography}{99}
\bibitem{ABO}
C. Abbott, S. Balasubramanya and D. Osin, {\sl Hyperbolic structures on groups}, Algebr. Geom. Topol., 19-4 (2019), 1747--1835.

\bibitem{ABD}
C. Abbott, J. Behrstock, M. Durham, \emph{Largest acylindrical actions and stability in hierarchically hyperbolic groups}, arXiv:1705.06219.

\bibitem{AD}
C. Abbott, F. Dahmani, \emph{Acylindrically hyperbolic groups have property $P_{naive}$}, Math. Z. {\bf 291} (2019), 555--568.

\bibitem{Alg}
Y. Algom-Kfir, \emph{Strongly contracting geodesics in outer space}, Geom. Topol. {\bf 15} (2011), no. 4, 2181--2233. 

\bibitem{AntCum}
Y. Antol\'in, M. Cumplido, \emph{Parabolic subgroups acting on the additional length graph}, arXiv:1906.06325.

\bibitem{AMST}
Y. Antolín, M. Mj, A. Sisto, S. Taylor, \emph{Intersection properties of stable subgroups and bounded cohomology}, Indiana Univ. Math. J. {\bf 68} (2019), no. 1, 179--199. 


\bibitem{AMS}
Y. Antolín, A. Minasyan, A. Sisto, \emph{Commensurating endomorphisms of acylindrically hyperbolic groups and applications}, Groups Geom. Dyn. {\bf 10} (2016), no. 4, 1149--1210.


\bibitem{AouDurTay}
T. Aougab, M. Durham, S. Taylor, \emph{Pulling back stability with applications to $Out(F_n)$ and relatively hyperbolic groups}, J. Lond. Math. Soc. (2) {\bf 96} (2017), no. 3, 565--583. 

\bibitem{Arz}
G. Arzhantseva, \emph{On quasiconvex subgroups of word hyperbolic groups}, Geom. Dedicata {\bf 87} (2001), no. 1-3, 191--208.


\bibitem{ACGH}
G. Arzhantseva, C. Cashen, D. Gruber, D. Hume, \emph{Characterizations of Morse quasi-geodesics via superlinear divergence and sublinear contraction}, Doc. Math. {\bf 22} (2017), 1193--1224. 

\bibitem{ArzMin}
G. Arzhantseva, A. Minasyan, \emph{Relatively hyperbolic groups are $C*$-simple}, J. Funct. Anal. {\bf 243} (2007), no. 1, 345--351. 

\bibitem{ASW}
M. Aschenbrenner, S. Friedl,  H. Wilton, \emph{3-manifold groups}, EMS Series of Lectures in Mathematics. European Mathematical Society (EMS), Zürich, 2015. xiv+215 pp. ISBN: 978-3-03719-154-5. 

\bibitem{BHS1}
J. Behrstock, M. Hagen, A. Sisto, \emph{Hierarchically hyperbolic spaces, I: Curve complexes for cubical groups}, Geom. Topol. {\bf 21} (2017), no. 3, 1731--1804. 
\bibitem{BHS2}
J. Behrstock, M. Hagen, A. Sisto,  \emph{Hierarchically hyperbolic spaces II: Combination theorems and the distance formula},  Pacific J. Math. {\bf 299} (2019), no. 2, 257--338

\bibitem{BesFei14}
M. Bestvina, M. Feighn, \emph{Hyperbolicity of the complex of free factors} Adv. Math. {\bf 256} (2014), 104--155. 


\bibitem{BestvinaFujiwara}
M. Bestvina, K. Fujiwara, \emph{Bounded cohomology of subgroups of mapping class groups}, Geom. \& Top. {\bf 6} (2002), 69-89.

\bibitem{Bow08}
B. Bowditch, \emph{Tight geodesics in the curve complex}, Invent. Math. {\bf 171} (2008), no. 2, 281--300. 

\bibitem{BH}
M. Bridson, A.  Haefliger, \emph{Metric spaces of non-positive curvature},
Grundlehren der Mathematischen Wissenschaften [Fundamental Principles of Mathematical Sciences], {\bf 319}. Springer-Verlag, Berlin, 1999. xxii+643 pp.

\bibitem{CalMah}
D. Calegari, J. Maher, \emph{Statistics and compression of scl}, Ergodic Theory Dynam. Systems {\bf 35} (2015), no. 1, 64--110. 

\bibitem{CalvezWiest1}
M. Calvez,  B. Wiest. \emph{Curve  graphs  and  Garside  groups}, Geom.  Dedicata, {\bf 188} (2016), no.1, 195--213.

\bibitem{CalvezWiest2}
M. Calvez,  B. Wiest \emph{Acylindrical  hyperbolicity  and  Artin-Tits  groups  of spherical type}, Geom. Dedicata, {\bf 191} (2017), no. 1, 199--215.

\bibitem{Cha}
V. V. Chaynikov, Properties of hyperbolic groups: free normal subgroups, quasiconvex subgroups and actions of maximal growth, Ph.D. Thesis, Vanderbilt University, (2012), available at http://etd.library.vanderbilt.edu/available/etd-06212012-172048/unrestricted/CHAYNIKOV.pdf.

\bibitem{CorHum}
M. Cordes, D. Hume, \emph{Stability and the Morse boundary}, J. Lond. Math. Soc. (2) {\bf 95} (2017), no. 3, 963–988.


\bibitem{DGO}
F. Dahmani, V. Guirardel, D. Osin, \emph{Hyperbolically embedded subgroups and rotating families in groups acting on hyperbolic spaces}, Mem. Amer. Math. Soc. {\bf 245} (2017), no. 1156, v+152 pp.

\bibitem{DH} 
F. Dahmani, C. Horbez, \emph{Spectral theorems for random walks on mapping class groups and $\operatorname{Out}(\mathbb F_n)$}, Int. Math. Res. Not. IMRN {\bf 9} (2018), 2693--2744. 

\bibitem{DowTay}
S. Dowdall, S. Taylor,  \emph{Hyperbolic extensions of free groups}, Geom. Topol. {\bf 22} (2018), no. 1, 517--570. 

\bibitem{DT}
M. Durham, S. Taylor, \emph{Convex cocompactness and stability in mapping class groups}, Algebr. Geom. Topol. {\bf 15} (2015), no. 5, 2839--2859. 

\bibitem{Primer}
B. Farb, D. Margalit, \emph{A primer on mapping class groups}, 
Princeton Mathematical Series, {\bf 49}. Princeton University Press, Princeton, NJ, 2012. xiv+472 pp. ISBN: 978-0-691-14794-9. 


\bibitem{FarMos}
B. Farb, L. Mosher, \emph{Convex cocompact subgroups of mapping class groups}, Geom. Topol. {\bf 6} (2002), 91--152. 

\bibitem{GouShc}
S. Gouëzel, Sébastien, V Shchur, \emph{A corrected quantitative version of the Morse lemma}, J. Funct. Anal. {\bf 277} (2019), no. 4, 1258–1268. 

\bibitem{Gro}
M. Gromov, \emph{Hyperbolic groups}, Essays in group theory, 75--263,
Math. Sci. Res. Inst. Publ., {\bf 8}, Springer, New York, 1987.

\bibitem{Ham05}
U. Hamenst{\"a}dt, \emph{Word hyperbolic extensions of surface groups}, arXiv:0807.4891v2.

\bibitem{HamHen18}
U. Hamenst{\"a}dt, S. Hensel, \emph{Stability in outer space},  Groups Geom. Dyn. {\bf 12} (2018), no. 1, 359--398. 

\bibitem{HanMos}
M. Handel, L. Mosher \emph{Subgroup decomposition in $Out(F_n)$: Introduction and Research Announcement}, arXiv:1302.2681.

\bibitem{Hruska}
C. Hruska, \emph{Relative hyperbolicity and relative quasiconvexity for countable groups} Algebr. Geom. Topol. {\bf 10} (2010), no. 3, 1807--1856. 

\bibitem{HruWis} 
C. Hruska, D. Wise, \emph{Packing subgroups in relatively hyperbolic groups}, Geom. Topol. {\bf 13} (2009), no. 4, 1945--1988. 

\bibitem{H16}
 M. Hull \emph{Small cancellation in acylindrically hyperbolic groups}, Groups Geom. Dyn. {\bf 10} (2016), no. 4, 1077--1119.

\bibitem{HO16}
M. Hull, D. Osin, \emph{Transitivity degrees of countable groups and acylindrical hyperbolicity}, Israel J. Math. {\bf 216} (2016), no. 1, 307--353. 

\bibitem{Iva92}
A. Ivanov, \emph{Subgroups of Teichmuller modular groups}, Translations of Math. Monographs,
{\bf 115}, Amer. Math. Soc. 1992.

\bibitem{KenLei}
A. Kent, C. Leininger, \emph{Shadows of mapping class groups: capturing convex cocompactness}, Geom. Funct. Anal. {\bf 18} (2008), no. 4, 1270--1325. 

\bibitem{KimKob13}
S. Kim, T. Koberda, \emph{Embedability between right-angled Artin groups}. Geom. Topol. {\bf 17} (2013), no. 1, 493--530. 

\bibitem{KimKob14}
S. Kim, T. Koberda,  \emph{The geometry of the curve graph of a right-angled Artin group}, Internat. J. Algebra Comput. {\bf 24} (2014), no. 2, 121--169. 


\bibitem{KobManTay}
T. Koberda, J. Mangahas, S. Taylor, \emph{The geometry of purely loxodromic subgroups of right-angled Artin groups} Trans. Amer. Math. Soc. {\bf 369} (2017), no. 11, 8179--8208.
 
\bibitem{ManMP}
J. Manning, E. Mart\'inez-Pedroza, \emph{Separation of relatively quasiconvex subgroups} Pacific J. Math. {\bf 244} (2010), no. 2, 309--334. 

\bibitem{MP}
E. Mart\'inez-Pedroza \emph{Combination of quasiconvex subgroups in relatively hyperbolic groups}, Thesis (Ph.D.)-The University of Oklahoma. 2008. 79 pp. ISBN: 978-0549-50182-4.

\bibitem{MasurMinsky}
H. Masur, Y. Minsky, \emph{Geometry of the complex of curves I: Hyperbolicity},  Invent. Math. {\bf 138} (1999), no. 1, 103--149.

\bibitem{MasurMinskyII}
H. Masur, Y. Minsky, \emph{Geometry of the complex of curves II: Hierarchical structure}, Geom. Funct. Anal. {\bf 10} (2000), no. 4, 902--974.

\bibitem{MS}
J. Maher, A. Sisto, \emph{Random subgroups of acylindrically hyperbolic groups and hyperbolic embeddings}, International Mathematics Research Notices, to appear, arXiv:1701.00253.

\bibitem{MT}
J. Maher, G. Tiozzo, \emph{Random walks on weakly hyperbolic groups},  J. Reine Angew. Math. {\bf 742} (2018), 187--239. 

\bibitem{MaherTiozzo18}
J. Maher, G. Tiozzo, \emph{Random walks, WPD actions, and the Cremona group}, arXiv:1807.10230.
 
\bibitem{MasMin1}
H. Masur, Y. Minsky, \emph{Geometry of the complex of curves. I. Hyperbolicity},  Invent. Math. {\bf 138} (1999), no. 1, 103--149.  
 
 \bibitem{Minasyan}
 A. Minasyan, \emph{On residualizing homomorphisms preserving quasiconvexity}, Comm. Algebra. {\bf 33} (2005), no. 7, 2423–2463.
 
\bibitem{Minasyan2}
A. Minasyan, \emph{Some properties of subsets of hyperbolic groups}, Comm. Algebra {\bf 33} (2005), no. 3, 909–935. 

\bibitem{Minasyan3} \emph{On residual properties of word hyperbolic groups}, J. of Group Theory {\bf 9} (2006), No. 5, pp. 695--714.

\bibitem{MinOsi}
A. Minasyan, D. Osin, \emph{Acylindrical hyperbolicity of groups acting on trees}.  Math. Ann. {\bf 362} (2015), no. 3--4, 1055--1105. 

 \bibitem{Minsky}
 Y. Minsky, \emph{Quasi-projections in Teichmuller space}, J. Reine Angrew. Math., {\bf 473} (1996), 121--136.
 
\bibitem{Ols93}
A.Yu. Ol'shanskii, \emph{On residualing homomorphisms and
$G$--subgroups of hyperbolic groups}, Internat. J.
Algebra Comput. {\bf 3}
(1993), 4, 365--409.

 \bibitem{Olshanskii}
 A.Yu. Ol'shanskii, \emph{Periodic quotients of hyperbolic groups}, Mat. Zbornik {\bf 182} (1991), no. 4, 543--567.
 
 \bibitem{Osin:RelHyp}
 D. Osin, \emph{Relatively hyperbolic groups: intrinsic geometry, algebraic properties, and algorithmic problems}, Mem. Amer. Math. Soc. {\bf 179} (2006), no. 843, vi+100 pp. 
 
\bibitem{OsinAH}
D. Osin, \emph{Acylindrically hyperbolic groups}, Trans. Amer. Math. Soc. {\bf 368} (2016), no. 2, 851--888.

 
 \bibitem{Rafi}
 K. Rafi, \emph{Hyperbolicity in Teichmuller space}, Geom. \& Top. {\bf 18} (2014) 3025--3053.
 
 
 \bibitem{RST}
 J. Russell, D. Spriano, H. Tran, \emph{The local-to-global property for Morse quasi-geodesics}, arXiv:1908.11292.
 
 
\bibitem{Sis12}
A. Sisto, \emph{On metric relative hyperbolicity}, arXiv:1210.8081. 

 \bibitem{Sisto}
 A. Sisto, \emph{Quasi-convexity of hyperbolically embedded subgroups}  Math. Z. {\bf 283} (2016), no. 3--4, 649--658. 
 
 \bibitem{SisHHS}
 A. Sisto, \emph{What is a hierarchically hyperbolic space?}, arXiv:1707.00053.


 \bibitem{TaylorTiozzo}
S. Taylor, G. Tiozzo, \emph{Random extensions of free groups and surface groups are hyperbolic} Int. Math. Res. Not. IMRN 2016, no. 1, 294--310. 

\bibitem{Tran}
H. C. Tran, \emph{On strongly quasiconvex subgroups}, Geom. Topol. {\bf 23} (2019), no. 3, 1173--1235.

\bibitem{vogtmann:CVnsurvey}
K. Vogtmann, \emph{On the geometry of outer space}, Bull. Amer. Math. Soc. {\bf 52} (2014), no. 1, 27--46.

\bibitem{WilZal}
 H. Wilton, P. Zalesskii, \emph{Profinite properties of graph manifolds}, Geom. Dedicata {\bf 147} (2010), 29--45.

\end{thebibliography}
